\documentclass[11pt,reqno]{amsart} 

\usepackage[utf8]{inputenc} 
\usepackage{bbm}

\usepackage[margin=1in]{geometry} 

\usepackage{graphicx} 

 \usepackage[parfill]{parskip} 
 
\usepackage{booktabs} 
\usepackage{array} 
\usepackage{paralist} 
\usepackage{verbatim} 
\usepackage{subfig} 
\usepackage{mathrsfs}
\usepackage{amssymb}
\usepackage{amsthm}
\usepackage{amsmath,amsfonts,amssymb,esint}
\usepackage{graphics,color}
\usepackage{enumerate}
\usepackage{mathtools,centernot}
\usepackage{cases}

\usepackage{xfrac}

\newtheorem{theorem}{Theorem}[section]
\newtheorem{lemma}[theorem]{Lemma}
\newtheorem{corollary}[theorem]{Corollary}
\newtheorem{definition}[theorem]{Definition}
\newtheorem{proposition}[theorem]{Proposition}
\newtheorem{remark}[theorem]{Remark}

\newtheorem*{remark 1}{Remark}

\numberwithin{equation}{section}

\newcommand{\norm}[1]{\left\|#1\right\|}
\newcommand{\abs}[1]{\left|#1\right|}

\DeclarePairedDelimiter{\ceil}{\lceil}{\rceil}

\newcommand{\T}{\ensuremath{\mathbb{T}}}
\newcommand*{\R}{\ensuremath{\mathbb{R}}}
\renewcommand*{\S}{\ensuremath{\mathcal{S}}}

\newcommand*{\N}{\ensuremath{\mathbb{N}}}
\newcommand*{\Z}{\ensuremath{\mathbb{Z}}}
\newcommand*{\C}{\ensuremath{\mathbb{C}}}
\newcommand{\eps}{\varepsilon}
\newcommand*{\tr}{\ensuremath{\mathrm{tr\,}}}

\newcommand*{\Id}{\ensuremath{\mathrm{Id}}}

\newcommand*{\RR}{\ensuremath{\mathcal{R}}}
\newcommand*{\RRc}{\ensuremath{\mathcal{B}}}

\usepackage{color, graphicx}
\usepackage{mathrsfs, dsfont}

\usepackage{hyperref}

\def\div{\mathop{\rm div}\nolimits}    
\def\supp{\mathop{\rm supp}\nolimits}    
\def\curl{\mathop{\rm curl}\nolimits}    

\title{Onsager's conjecture for admissible weak solutions}

\author{Tristan Buckmaster}
\address{Courant Insitute of Mathematical Sciences, New York University, New York, NY 10012, USA}
\email{buckmaster@cims.nyu.edu}

\author{Camillo De Lellis}
\address{Institut f\"ur Mathematik, Universit\"at Z\"urich, CH-8057 Z\"urich, Switzerland}
\email{camillo.delellis@math.unizh.ch}

\author{L\'aszl\'o Sz\'ekelyhidi Jr.}
\address{Institut f\"ur Mathematik, Universit\"at Leipzig, D-04103 Leipzig, Germany}
\email{laszlo.szekelyhidi@math.uni-leipzig.de}

\author{Vlad Vicol}
\address{Department of Mathematics, Princeton University, Princeton, NJ 08544, USA}
\email{vvicol@math.princeton.edu}

\date{\today}
\begin{document}

\begin{abstract}
We prove that given any $\beta<1/3$, a time interval $[0,T]$, and given any  smooth energy profile $e \colon [0,T] \to (0,\infty)$, there exists a weak solution $v$ of the three-dimensional Euler equations such that $v \in C^{\beta}([0,T]\times \T^3)$, with $e(t) = \int_{\T^3} |v(x,t)|^2 dx$ for all $t\in [0,T]$. Moreover, we show that a suitable $h$-principle holds in the regularity class $C^\beta_{t,x}$, for any $\beta<1/3$. The implication of this is that the dissipative solutions we construct are in a sense typical in the appropriate space of subsolutions as opposed to just isolated examples.
\end{abstract}

\maketitle


\section{Introduction}

In this paper we consider the incompressible Euler equations 
\begin{equation}\label{e:eulereq}
\left\{\begin{array}{l}
\partial_t v+ v\cdot\nabla v +\nabla p =0\\ \\
\div v = 0,
\end{array}\right.
\end{equation}
in the periodic setting $x\in\T^3=\R^3\setminus \Z^3$,
where $v$ is a vector field representing the velocity of the fluid and $p$ is the pressure. We study weak (distributional) solutions $v$ which are H\"older continuous in space, i.e.~such that\footnote{The smallest constant $C$ satisfying \eqref{e:holder-space} will be denoted by $[v]_\beta$, cf. Appendix \ref{s:hoelder}. We will write $v\in C^\beta (\T^3\times [0,T])$ when $v$ is H\"older continuous in the whole space-time.}
\begin{equation}\label{e:holder-space}
|v(x,t)-v(y,t)|\leq C|x-y|^{\beta}\quad\textrm{ for all }t\in[0,T]
\end{equation}
for some constant $C$ which is independent of time $t$. 

In his famous 1949 note on statistical hydrodynamics Lars Onsager~\cite{On1949} conjectured that the threshold regularity for the validity of the energy conservation of weak solutions of \eqref{e:eulereq} is the exponent $\sfrac{1}{3}$: in particular he asserted that for larger H\"older exponents any weak solution would conserve the energy, whereas for any smaller exponent there are solutions which do not. The first assertion was fully proved by Constantin, E and Titi in \cite{CoETi1994}, after a partial result of Eyink in \cite{Eyink} (see also~\cite{CCFS2008} for a sharper criterion in $L^3$-based spaces). Concerning the second assertion, the first proof of the existence of a square summable weak solution which does not preserve the energy is due to Scheffer in his pioneering paper \cite{Scheffer}. A different proof has been later given by Shnirelman in \cite{Shnirelman}. In \cite{DlSzAnnals} the second and third author realized that techniques from the theory of differential inclusions could be applied very efficiently to produce bounded weak solutions which violate the energy conservation in several forms. Pushed by the analogy of these constructions with the famous $C^1$ solutions of Nash and Kuiper for the isometric embedding problem (cf. \cite{Nash} and \cite{Kuiper}) they proposed to approach the remaining statement of the Onsager's conjecture in a similar way (cf. \cite{DlSz2012}). Indeed in \cite{DlSz2013} and \cite{DlSzJEMS} they were able to give the first examples of, respectively, continuous and H\"older continuous solutions which dissipate the energy, reaching the threshold exponent $\sfrac{1}{10}$. After a series of important partial results improving the threshold and the techniques from several points of view, cf. \cite{Is2013,BuDLeSz2013,BuDLeIsSz2015,Bu2015,BuDLeSz2016}, in his recent paper Isett \cite{Is2016} has been able to finally reach the Onsager exponent $\sfrac{1}{3}$. The proof of Isett combines previous ideas with two new important ingredients, one developed by Daneri and the third author in \cite{DaSz2016} (the introduction of Mikado flows, see Section \ref{s:perturbation_outline}) and one introduced by Isett himself the aforementioned paper (the gluing technique, see Section \ref{s:gluing_outline}). 

However, the solutions produced in \cite{Is2016} are only shown to be nonconservative and in fact for those solutions the total kinetic energy fails to be monotonic on any interval of time. Thus Isett's theorem left open the question whether it is possible or not to construct solutions which {\em dissipate} the kinetic energy (i.e. with strictly monotonic decreasing energy). In fact the latter is a relevant point, for at least two reasons: first of all because dissipative solutions satisfy the weak-strong uniqueness property \cite{LionsBook,BrenierDLSz2011} and secondly because indeed Onsager in his work conjectures the existence of dissipative solutions. Indeed, in \cite{On1949} Onsager states: {\it 
\begin{center}
It is of some interest to note that in principle, turbulent dissipation as described could take place just as readily without the final assistance by viscosity.
\end{center}
}

In this note we suitably modify the approach of Isett in order to show the following theorem.

\begin{theorem}\label{t:energy_profile}
Assume $e: [0, T]\to \R$ is a \emph{strictly positive} smooth function. Then for any $0<\beta<\sfrac 13$ there exists a weak solution $v\in C^{\beta}(\T^3\times [0,T])$ to \eqref{e:eulereq} such that 
\begin{equation*}
\int_{\T^3} \abs{v(x,t)}^2~dx=e(t)\,.
\end{equation*}
\end{theorem}

We are indeed able to prove a stronger statement than Theorem \ref{t:energy_profile}, namely an $h$-principle in the sense of \cite{DaSz2016}. Following \cite{DaSz2016} we introduce smooth strict subsolutions of the Euler equations.

\begin{definition} A smooth strict subsolution of \eqref{e:eulereq} on $\T^3\times [0,T]$ is a smooth triple $(\bar{v},\bar{p},\bar{R})$ with $\bar{R}$ a symmetric $2$-tensor, such that
\begin{equation}\label{e:ER}
\left\{\begin{array}{l}
\partial_t\bar{v}+\div(\bar{v}\otimes\bar{v})+\nabla\bar{p}=-\div\bar{R}\\ \\
\div\bar{v}=0,
\end{array}\right.
\end{equation}
and $\bar{R}(x,t)$ is positive definite for all $(x,t)$. 
\end{definition}

We then can prove that any smooth strict subsolution can be suitably approximated by $C^\beta$ solutions for any $\beta < \sfrac{1}{3}$. More precisely

\begin{theorem}\label{t:h-principle}
Let $(\bar{v},\bar{p},\bar{R})$ be a smooth strict subsolution of the Euler equations on $\T^3\times[0,T]$ and let $\beta<1/3$. Then there exists a sequence $(v_k,p_k)$ of weak solutions of \eqref{e:eulereq} such that $v_k\in C^{\beta}(\T^3\times [0,T])$, 
\begin{equation*}
v_k\overset{*}{\rightharpoonup }\bar{v}\quad\textrm{ and }\quad v_k\otimes v_k\overset{*}{\rightharpoonup }\bar{v}\otimes\bar{v} + \bar R\quad \mbox{in} \quad L^{\infty}
\end{equation*}
uniformly in time, and furthermore for all $t\in [0,T]$
\begin{equation}\label{e:h-principle_energy}
\int_{\T^3}|v_k|^2\,dx=\int_{\T^3} \left( |\bar{v}|^2+\tr\bar{R} \right)\,dx.
\end{equation}
\end{theorem}

Theorem \ref{t:energy_profile} can be concluded as a simple corollary of Theorem \ref{t:h-principle}. However we give an alternative simpler and self-contained argument for Theorem \ref{t:energy_profile}: indeed the proof of Theorem \ref{t:h-principle} invokes some results of \cite{DaSz2016}, whereas the argument for Theorem \ref{t:energy_profile} is entirely contained in our note, aside from technical propositions which are classical statements in the literature, all collected in the Appendix. 

\medskip

The most important differences in our proof compared to that of \cite{Is2016} rely on the estimates for the ``gluing step'' of Isett's proof (we refer to Section \ref{s:gluing_outline} for more details) and in a simple remark concerning the regions where the perturbation is added (see Section \ref{s:perturbation_outline}). We note that, even without the extra benefit of imposing the energy profile and achieving the most general $h$-principle statement, the proof proposed here is considerably shorter than that of \cite{Is2016}. 

\subsection*{Acknowledgments} 
The work of T.B. has been partially supported by the National Science Foundation grant DMS-1600868.
The research of C.D.L. has been supported by the grant 200021$\_$159403 of the Swiss National Foundation.
L.Sz.  gratefully acknowledges the support of the ERC Grant Agreement No. 277993.
V.V. was partially supported by the National Science Foundation grant DMS-1514771 and by an
Alfred P. Sloan Research Fellowship.

\section{Outline of the proof} As already mentioned, although Theorem \ref{t:energy_profile} can be recovered as a corollary of Theorem \ref{t:h-principle}, in this section we outline an independent proof, reducing it to a suitable iterative procedure, summarized in
Proposition~\ref{p:main} below. The same iteration procedure can be used to prove Theorem~\ref{t:h-principle}, as
shown in Section \ref{s:h} at the end of the note, but the corresponding argument we will need some results from~\cite{DaSz2016}, which we state without proof. In contrast, the proof of Theorem \ref{t:energy_profile} is completely self-contained.

\subsection{Inductive proposition}
First of all, we impose for the moment that
\begin{equation}\label{e:energy_time_D}
\sup_{t\in [0,T]} \abs{\tfrac {d}{dt} e(t)}\leq 1\, 
\end{equation}
(we will see later that this can be done without loosing generality). 

Let then $q\geq 0$ be a natural number. At a given step $q$ we assume to have a triple $(v_q,p_q,\mathring{R}_q)$ to the Euler-Reynolds system \eqref{e:ER}, namely such that
\begin{equation}\label{e:euler_reynolds}
\left\{\begin{array}{l}
\partial_t v_q + \div (v_q\otimes v_q) + \nabla p_q =\div\mathring{R}_q\\ \\
\div v_q = 0\, ,
\end{array}\right.
\end{equation} 
to which we add the constraints that 
\begin{equation}\label{e:trace_free}
\tr \mathring{R}_q=0
\end{equation}
and that
\begin{equation}\label{e:press_const}
\int_{\T^3} p_q (x,t)\, dx = 0\, 
\end{equation}
(which uniquely determines the pressure). 

The size of the approximate solution $v_q$ and the error $\mathring{R}_q$ will be measured by a frequency $\lambda_q$ and an amplitude $\delta_q$, which are given by
\begin{align}
\lambda_q&= 2\pi \ceil{a^{(b^q)}}\label{e:freq_def}\\
\delta_q&=\lambda_q^{-2\beta} \label{e:size_def}
\end{align}
where $\ceil{x}$ denotes the smallest integer $n\geq x$, $a>1$ is a  large parameter, $b>1$ is close to $1$ and $0<\beta<\sfrac13$ is the exponent of Theorem \ref{t:energy_profile}. The parameters $a$ and $b$ are then related to $\beta$. 

We proceed by induction, assuming the estimates
\begin{align}
\norm{\mathring R_q}_{0}&\leq  \delta_{q+1}\lambda_q^{-3\alpha}\label{e:R_q_inductive_est}\\
\norm{v_q}_1&\leq M \delta_q^{\sfrac12}\lambda_q\label{e:v_q_inductive_est}\\
\norm{v_q}_0 & \leq 1- \delta_q^{\sfrac12}\label{e:v_q_0}\\
\delta_{q+1}\lambda_q^{-\alpha} &\leq e(t)-\int_{\T^3}\abs{  v_q}^2\,dx\leq \delta_{q+1}\label{e:energy_inductive_assumption}
\end{align}
where $0 < \alpha  < 1$ is a small parameter to be chosen suitably (which will depend upon $\beta$), and $M$ is a universal constant (which is fixed throughout the iteration and whose choice depends on certain geometric properties of the space of symmetric matrices and on the ``squiggling'' regions of the perturbation step, cf. Remark \ref{r:choice_of_M}, Lemma \ref{l:choice_of_M} and Definition \ref{d:choice_of_M}). We refer to Appendix~\ref{s:hoelder} for the definitions of the H\"older norms used above, where we take into account only {\em space regularity}. 

\begin{proposition}\label{p:main} There is a universal constant $M$ with the following property. Assume $0<\beta<\sfrac13$ and
\begin{equation}\label{e:b_beta_rel}
1<b<\frac{1-\beta}{2\beta}\,.
\end{equation}
Then there exists an $\alpha_0$ depending on $\beta$ and $b$, such that for any $0<\alpha<\alpha_0$ there exists an $a_0$ depending on $\beta$, $b$, $\alpha$ and $M$, such that for any $a\geq a_0$ the following holds: Given a strictly positive energy function $e: [0, T]\to \R$ satisfying \eqref{e:energy_time_D}, and a triple $(v_q,\mathring R_q, p_q)$ solving \eqref{e:euler_reynolds}-\eqref{e:press_const} and satisfying the estimates \eqref{e:R_q_inductive_est}--\eqref{e:energy_inductive_assumption}, then there exists a solution  $(v_{q+1}, \mathring R_{q+1}, p_{q+1})$ to \eqref{e:euler_reynolds}-\eqref{e:press_const} satisfying \eqref{e:R_q_inductive_est}--\eqref{e:energy_inductive_assumption} with $q$ replaced by $q+1$. Moreover, we have 
\begin{equation}
\norm{v_{q+1}-v_q}_0+\frac{1}{\lambda_{q+1}}\norm{v_{q+1}-v_q}_1 
\leq M\delta_{q+1}^{\sfrac12}\label{e:v_diff_prop_est}.
\end{equation}
\end{proposition}

The proof of Proposition~\ref{p:main} is summarized in the Sections \ref{s:stages}, \ref{s:mollification}, \ref{s:gluing_outline} and \ref{s:perturbation_outline}, but its details will occupy most of the paper and will be completed in Section~\ref{sec:p_main} below. We show next that this proposition immediately implies Theorem~\ref{t:energy_profile}.

\subsection{Proof of Theorem \ref{t:energy_profile}} First of all, we fix any H\"older exponent $\beta< \sfrac{1}{3}$ and also the
parameters $b$ and $\alpha$, the first satisfying \eqref{e:b_beta_rel} and the second smaller than the threshold given in Proposition \ref{p:main}. Next we show that, without loss of generality, we may further assume the energy profile satisfies 
\begin{equation}\label{e:normalized_energy}
\inf_t e(t) \geq \delta_{1}\lambda_0^{-\alpha},\qquad\sup_t e(t)\leq \delta_{1}, \quad\mbox{and}\quad \sup_t e'(t)\leq 1\, ,
\end{equation}
provided the parameter $a$ is chosen sufficiently large. 

To see this, we first note that the Euler equations are invariant under the transformation
\[
v(x,t)\mapsto \Gamma v(x,\Gamma t) \quad \mbox{and} \quad p(x,t) \mapsto \Gamma^2 p(x,\Gamma t)\,.
\]
Thus if we choose 
\[
\Gamma=\left(\frac{\delta_{1}}{\sup_t e(t)}\right)^{\sfrac12},
\] 
then using the scaling invariance, the stated problem  reduces to finding a solution with the energy profile given by
\[
\tilde e(t)=\Gamma^2 e(t)\,,
\]
for which we have
\[
\inf_t \tilde e(t) \geq \frac{\delta_{1}\inf_t e(t)}{\sup_t e(t)}, \qquad
\sup_t \tilde e(t)\leq  \delta_{1}, \qquad \mbox{and}\qquad
\sup_t\tilde e'(t)\leq \left(\frac{\delta_{1}}{\sup_t e(t)}\right)^{\sfrac32} \sup_t e'(t).
\]
If $a$ is chosen sufficiently large then we can ensure
\[\sup_t\tilde e'(t)\leq \left(\frac{\delta_{1}}{\sup_t e(t)}\right)^{\sfrac32}  \sup_t e'(t)\leq 1, \qquad\mbox{and}\qquad \frac{\inf_t e(t)}{\sup_t e(t)}\geq \lambda_0^{-\alpha}\,.\]

Now we apply Proposition \ref{p:main} iteratively with $(v_0,R_0, p_0)=(0,0,0)$. Indeed the pair $(v_0, R_0)$ trivially satisfies \eqref{e:R_q_inductive_est}--\eqref{e:v_q_0}, whereas the estimate \eqref{e:energy_inductive_assumption} and \eqref{e:energy_time_D} follows as a consequence of \eqref{e:normalized_energy}. 
Notice that by \eqref{e:v_diff_prop_est} $v_q$ converges uniformly to some continuous $v$. Moreover, we recall that the pressure
is determined by
\begin{equation}\label{e:pressure_eq}
\Delta p_q = \div \div (-v_q\otimes v_q + \mathring{R}_q)
\end{equation}
and \eqref{e:press_const} and thus $p_q$ is also converging to some pressure $p$ (for the moment only in $L^r$ for every $r<\infty$). Since $\mathring{R}_q\to 0$ uniformly, the pair $(v,p)$ solves the Euler equations. 

Observe that using \eqref{e:v_diff_prop_est} we also infer\footnote{Throughout the manuscript we use the the notation $x \lesssim y$ to denote  $x \leq C y$, for a sufficiently large constant $C>0$, which is independent of $a,b$, and $q$, but may change from line to line.} 
\begin{align*}
\sum_{q=0}^{\infty} \norm{v_{q+1}-v_q}_{\beta'} &
\lesssim \; \sum_{q=0}^{\infty} \norm{v_{q+1}-v_q}_{0}^{1-\beta'}\norm{v_{q+1}-v_q}_{1}^{\beta'}
\lesssim \; \sum_{q=0}^{\infty} \delta_{q+1}^{\frac{1-\beta'}2}\left(\delta_{q+1}^{\sfrac12}\lambda_q\right)^{\beta'} \notag
\lesssim  \sum_{q=0}^{\infty} \lambda_q^{\beta'-\beta}
\end{align*}
and hence that $v_q$ is uniformly bounded in $C^0_tC^{\beta'}_x$ for all $\beta'<\beta$. To recover the time regularity, we could use the Euler equations and the general result in \cite{Isett-time}.  Nevertheless, we believe that the following short and self-contained proof of the time-regularity may be of independent interest: 

Fix a smooth  standard mollifier $\psi$ {\em in space}, let $q \in \N$,  and consider $\tilde v_q := v*\psi_{2^{-q}}$, where $\psi_{\ell}(x) = \ell^{-3} \psi(x \ell^{-1})$. From standard mollification estimates we have
\begin{align}
\norm{\tilde  v_q-v}_0\lesssim \norm{v}_{\beta'} 2^{-q \beta'},
\label{e:trivial:1}
\end{align}
 and thus $\tilde v_q - v \to 0$ uniformly as $q \to \infty$. Moreover, $\tilde v_q$ obeys the following equation
 \[
 \partial_t \tilde v_{q}+\div\left(v  \otimes v \right)*\psi_{2^{-q}}  +\nabla p * \psi_{2^{-q}}=0.
 \]
Next, since
\[
-\Delta p *\psi_{2^{-q}} = \div \div (  v\otimes v )* \psi_{2^{-q}} \, ,
\]
using Schauder's estimates, for any fixed $\varepsilon>0$ we get
\[
\|\nabla p * \psi_{2^{-q}}\|_0 \leq \|\nabla p * \psi_{2^{-q}}\|_\varepsilon \lesssim \|v\otimes v\|_{\beta'} 2^{q(1 + \eps - \beta')}
\lesssim \|v\|^2_{\beta'} 2^{q(1 + \eps - \beta')}\, ,
\]
(where the constant in the estimate depends on $\varepsilon$ but not on $q$). Similarly, 
\[
\norm{ \left(v  \otimes v \right)*\psi_{2^{-q}}}_1 \lesssim\norm{v\otimes v}_{\beta'}  2^{q(1-\beta')}  \lesssim \norm{v}_{\beta'}^2  2^{q(1-\beta')} \,.
\]
Thus the above estimates yield
\begin{align}
\norm{\partial_t \tilde  v_{q}}_0 &\lesssim  \|v\|^2_{\beta'} 2^{q(1 + \eps - \beta')}\,.
\label{e:trivial:2}
\end{align}
Next, for $\beta'' < \beta'$ we conclude from \eqref{e:trivial:1} and \eqref{e:trivial:2} that
\begin{align*}
\norm{\tilde v_q - \tilde v_{q+1}}_{C^0_x C^{\beta''}_t}  
&\lesssim \left( \norm{\tilde v_q - v}_0 + \norm{\tilde v_{q+1} - v}_0 \right)^{1-\beta''}
\left(\norm{\partial_t \tilde v_q}_0  + \norm{\partial_t \tilde v_{q+1}}_0\right)^{\beta''}
\notag\\
&\lesssim \norm{v}_{\beta'}^{1 + \beta''} 2^{- q \beta' (1-\beta'')} 2^{q \beta''(1+ \eps - \beta')} = \norm{v}_{\beta'}^{1 + \beta''} 2^{- q (\beta' - (1+ \eps) \beta'')} \notag\\
&\lesssim \norm{v}_{\beta'}^{1 + \beta''} 2^{-q \eps}
\end{align*}
Here we have chosen $\eps>0$ sufficiently small (in terms of $\beta'$ and $\beta''$) so that  that $\beta' - (1+ \eps) \beta'' \geq \eps $. Thus, the series 
\[
v = \tilde v_0 + \sum_{q\geq0} (\tilde v_{q+1} - \tilde v_q)
\]
converges in $C^0_x C^{\beta''}_t$. Since we already know $v \in C^0_t C^{\beta'}_x$, we obtain that $v \in C^{\beta''}([0,T] \times \T^3)$ as desired, with $\beta'' < \beta' < \beta < 1/3$ arbitrary.

Finally, since $\delta_{q+1}\rightarrow 0$ as $q\rightarrow \infty$, from \eqref{e:energy_inductive_assumption} we have
\begin{equation*}
\int_{\T^3}\abs{  v}^2\,dx= e(t)\, ,
\end{equation*}
which completes the proof of the theorem. 

\subsection{Stages}\label{s:stages}
Except for Section \ref{s:h}, the rest of the paper is devoted to the proof of Proposition \ref{p:main}. It will be useful to make the assumption that $\alpha$ is small enough so to have
\begin{align}\label{e:some_param_ineq}
\lambda_{q}^{3\alpha}\leq\left(\frac{\delta_q}{\delta_{q+1}}\right)^{\sfrac32}\leq \frac{\lambda_{q+1}}{\lambda_q}\, ,
\end{align}
which also require that $a$ is large enough to absorb any constant appearing from the ratio $\lambda_q/a^{(b^q)}$, for which we have the elementary bounds
\begin{equation}\label{e:bloody_integers}
2\pi \leq \frac{\lambda_q}{a^{b^q}}\leq 4\pi\, .
\end{equation}

The proof consists of three stages, in each of which we modify $v_q$. Roughly speaking, the stages are as follows:
\begin{itemize}
\item Mollification: $(v_q,\mathring{R}_q)\mapsto (v_\ell,\mathring{R}_\ell)$;
\item Gluing: $(v_\ell,\mathring{R}_\ell)\mapsto (\bar{v}_q,\mathring{\overline R}_q)$;
\item Perturbation: $(\bar{v}_q,\mathring{\overline R}_q)\mapsto (v_{q+1},\mathring{R}_{q+1})$.	
\end{itemize}

\subsection{Mollification step}\label{s:mollification}
The first stage is mollification: we mollify $v_q$ at length scale $\ell$ in order to handle the \emph{loss of derivative} problem, typical of convex integration schemes. To this aim, we fix a standard mollification kernel $\psi$ {\em in space} and introduce the mollification parameter
\begin{equation}\label{e:ell_def}
\ell:=\frac{\delta_{q+1}^{\sfrac 12}}{\delta_q^{\sfrac12 }\lambda_q^{1+\sfrac{3\alpha}{2}}}\, ,
\end{equation}
and define
\begin{align*}
v_{\ell}:=& v_q*\psi_\ell\\
\mathring{R}_{\ell}:=& \mathring R_q*\psi_\ell + (v_q\mathring\otimes v_q)*\psi_\ell - v_{\ell}\mathring\otimes v_{\ell} 
\end{align*}
where $f \mathring \otimes g$ is the traceless part of the tensor $f\otimes g$. These functions obey the equation
\begin{equation}\label{e:euler_reynolds_l}
\left\{\begin{array}{l}
\partial_t v_\ell + \div (v_\ell\otimes v_\ell) + \nabla p_\ell =\div\mathring{R}_\ell\\ \\
\div v_\ell = 0\, ,
\end{array}\right.
\end{equation} 
in view of \eqref{e:euler_reynolds}.

Observe, again choosing $\alpha$ sufficiently small and $a$ sufficiently large we can assume
\begin{equation}\label{e:compare-lambda-ell}
\lambda_q^{-3/2}\leq\ell \leq \lambda_q^{-1}\,,
\end{equation}
which will be applied repeatedly in order to simplify the statements of several estimates.

From \eqref{e:compare-lambda-ell}, standard mollification estimates and Proposition \ref{p:CET} we  obtain the following bounds\footnote{In the following, when considering higher order norms $\|\cdot\|_N$ or $\|\cdot\|_{N+1}$, the symbol $\lesssim$ will imply that the constant in the inequality might also depend on $N$.}

\begin{proposition}\label{p:est_mollification}
\begin{align}
\norm{v_{\ell}-v_q}_0&\lesssim \delta_{q+1}^{\sfrac12}\lambda_q^{-\alpha}\,,\label{e:v:ell:0}\\
\norm{v_{\ell}}_{N+1} &\lesssim \delta_q^{\sfrac 12}\lambda_q\ell^{-N}\qquad \forall N \geq 0\,,  \label{e:v:ell:k}\\
\norm{\mathring{R}_{\ell}}_{N+\alpha}&\lesssim  \delta_{q+1}\ell^{-N+\alpha} \qquad \forall N\geq 0\,. \label{e:R:ell}\\
\abs{\int_{\T^3}\abs{v_q}^2-\abs{v_{\ell}}^2\,dx} &\lesssim \delta_{q+1}\ell^{\alpha}\,.
\label{e:vq_vell_energy_diff}
\end{align}
\end{proposition}
\begin{proof}[Proof of Proposition~\ref{p:est_mollification}] 
The bounds \eqref{e:v:ell:0} and \eqref{e:v:ell:k} follow from the obvious estimates
\[
\|v_\ell - v_q\|_0 \leq \|v_q\|_1 \ell \lesssim \delta_q^{\sfrac{1}{2}} \lambda_q \ell \lesssim \delta_{q+1}^{\sfrac{1}{2}} \lambda_q^{-\alpha}\, 
\]
and
\[
\|v_\ell\|_{N+1}\leq \|v_q\|_1 \ell^{-N} \lesssim \delta_q^{\sfrac 12}\lambda_q\ell^{-N}\, .
\]
Next, applying Proposition \ref{p:CET},
\begin{align*}
\norm{\mathring{R}_{\ell}}_{N+\alpha}\lesssim &\|\mathring{R}_q\|_0 \ell^{-N-\alpha} + \|v_q\|_1^2 \ell^{2-N-\alpha}
\lesssim \delta_{q+1} \lambda_q^{-3\alpha} \ell^{-N-\alpha} + \delta_q \lambda_q^2 \ell^2 \ell^{-N-\alpha}
\lesssim \delta_{q+1} \lambda_q^{-3\alpha}  \ell^{-N-\alpha}\, ,
\end{align*}
on the other hand, by \eqref{e:compare-lambda-ell} $\lambda_q^{-3\alpha} \leq \ell^{2\alpha}$, from which \eqref{e:R:ell} follows. Similarly, by Proposition \ref{p:CET},
\begin{align*}
\abs{\int_{\T^3}\abs{v_q}^2-\abs{v_{\ell}}^2\,dx} = &\abs{\int_{\T^3}(\abs{v_q}^2)_\ell-\abs{v_{\ell}}^2\,dx} \lesssim \norm{(\abs{v_q}^2)_\ell-\abs{v_{\ell}}^2}_0\lesssim \norm{v_q}_1^2\ell^2\, ,
\end{align*}
which implies \eqref{e:vq_vell_energy_diff}. 
\end{proof}

\subsection{Gluing step}\label{s:gluing_outline}
In the second stage we encounter the new crucial ingredient introduced by Isett in \cite{Is2016}: we glue together exact solutions to the Euler equations in order to produce a new $\overline v_q$, close to $v_q$, whose associated Reynolds stress error has support in pairwise disjoint temporal regions of length $\tau_q$ in time, where
\begin{equation}\label{e:tau_def}
\tau_q= \frac{\ell^{2\alpha}}{\delta_q^{\sfrac 12}\lambda_q}.
\end{equation} 

The parameter $\tau_q$ should be compared to the parameter $\mu^{-1}$ used in the paper \cite{BuDLeIsSz2015}. Indeed, $\tau_q^{-1}$ satisfies precisely the same parameter inequalities that $\mu$ satisfies in Section 2 of \cite{BuDLeIsSz2015}. We note in particular that like in \cite{BuDLeIsSz2015} we have the CFL-like condition
\begin{equation}
\tau_q \norm{v_{\ell}}_{1+\alpha} \stackrel{\eqref{e:v:ell:k}}{\lesssim} \tau_q\delta_q^{\sfrac{1}{2}} \lambda_q \ell^{-\alpha}  \lesssim \ell^{\alpha} \ll 1 \label{e:CFL}
\end{equation}
as long as $a$ is sufficiently large.

More precisely, we aim to construct a new triple $(\overline v_q,  \mathring{\overline R_q},\overline p_q)$ solving the Euler Reynolds equation \eqref{e:euler_reynolds} such that the temporal support of $ \mathring{\overline R_q} $ is contained in pairwise disjoint intervals $I_i$ of length $\sim\tau_q$ and such that the gaps between neighbouring intervals is also of length $\sim\tau_q$. More precisely, for any $n\in\Z$ let 
$$
t_n=n\tau_q,\qquad I_n=[t_n+\tfrac{1}{3}\tau_q,t_n+\tfrac{2}{3}\tau_q]\cap [0,T],\qquad J_n=(t_n-\tfrac{1}{3}\tau_q,t_n+\tfrac{1}{3}\tau_q)\cap [0,T] \, .
$$
We require
\begin{equation}\label{e:Rq_support}
\supp \mathring{\overline R_q}  \subset \bigcup_{n\in \N} I_n\times\T^3.
\end{equation}

Moreover, $(\overline v_q,  \mathring{\overline R_q})$ 
will satisfy the following estimates for any $N\geq 0$
\begin{align}
\norm{\overline v_q-v_{\ell}}_0&\lesssim \delta_{q+1}^{\sfrac12}\ell^{\alpha} \label{e:overline_v} \\
\norm{\overline v_q}_{1+N} &\lesssim \delta_q^{\sfrac 12}\lambda_q \ell^{-N} \label{e:overline_v_N}\\
\norm{\mathring{\overline R_q}}_{N+\alpha} &\lesssim \delta_{q+1}\ell^{-N+\alpha}  \label{e:overline_R}\\
\norm{\partial_t \mathring{\overline R_q}+(\overline v_q\cdot \nabla) \mathring{\overline R_q}}_{N+\alpha} &\lesssim \delta_{q+1}\delta_q^{\sfrac12}\lambda_q\ell^{-N-\alpha}\label{e:Dt:overline_R}\\
\left| \int_{\T^3} |\bar v_q|^2 - |v_\ell |^2 dx \right| &\lesssim \delta_{q+1}\ell^\alpha\label{e:overline_energy}
\end{align}
where the implicit constants depend only on $M,\alpha$, and $N$, cf. Propositions \ref{p:vq:vell}, \ref{p:Rq} and
\ref{p:glued:energy}. 

The gluing procedure will be broken up into two parts: first, we construct a sequence of exact solutions to the Euler equations with appropriate stability estimates in Section \ref{s:stability} and then we glue the solutions together in Section \ref{s:gluing} with a partition of unity in order to construct $\overline v_q$ satisfying the properties mentioned above. This is indeed the key idea of Isett in \cite{Is2016}. The main difference with \cite{Is2016} is in the construction of the tensor $\mathring{\overline R_q}$: in this paper we use the usual elliptic operators introduced in \cite{DlSz2013}. This has the advantage that our Reynolds stress remains trace free, in contrast to the one of \cite{Is2016}, and in turn this is crucial to control the energy in the perturbation step below. It should be noticed that in \cite{Is2016} the author resorts to a different definition of $\mathring{\overline R_q}$ because he is not able to find efficient estimates. Our main technical improvement is that this difficulty can be overcome employing suitable commutator estimates on the advective derivative of differential operators 
of negative order, cf. the proof of Proposition \ref{p:S_est} and Proposition \ref{p:com:CZ:multiplication}. This remark allows us not only to keep a better control on the energy and a trace-free Reynolds stress with the desired estimate, but it also shortens the arguments considerably compared to~\cite{Is2016}.

\subsection{Perturbation and proof of Proposition \ref{p:main}}\label{s:perturbation_outline}

The gluing procedure can be used to localize the Reynolds stress error $\mathring{\overline R}_q$ to small disjoint temporal regions, but it cannot be used to completely eliminate the error. 

First of all note that as a corollary of \eqref{e:energy_inductive_assumption}, \eqref{e:vq_vell_energy_diff} and \eqref{e:overline_energy}, by choosing $a$ sufficiently large we can ensure that
\begin{equation}\label{e:voverline_energy_error}
\frac{\delta_{q+1}}{2\lambda_q^{\alpha}}\le e(t)-\int_{\T^3}\abs{\overline v_q}^2\,dx\leq 2\delta_{q+1}\,.
\end{equation}
Starting with the solution $(\overline{v}_q,\overline{p}_q,\mathring{\overline R_q})$ satisfying \eqref{e:Rq_support} and  the estimates \eqref{e:overline_v}-\eqref{e:voverline_energy_error}, we then produce a new solution $(v_{q+1},p_{q+1},\mathring{R}_{q+1})$ of the Euler-Reynolds system \eqref{e:euler_reynolds} with estimates
\begin{align}
\|v_{q+1}-\overline v_q\|_0 + \lambda_{q+1}^{-1} \|v_{q+1}-\overline v_q\|_1&\leq \frac{M}{2}\delta_{q+1}^{\sfrac12}\label{e:outline_v_diff}\\
\|\mathring{R}_{q+1}\|_\alpha &\lesssim \frac{\delta_{q+1}^{\sfrac12}\delta_q^{\sfrac12}\lambda_q}{\lambda_{q+1}^{1-4\alpha}}\,.\label{e:outline_R_est}\\
    \abs{e(t)-\int_{\T^3}\abs{v_{q+1}}^2\,dx-\frac{\delta_{q+2}}2 } &\lesssim \frac{\delta_q^{\sfrac12}\delta_{q+1}^{\sfrac12}\lambda_q^{1+ 2\alpha}}{\lambda_{q+1}}\, ,\label{e:outline_energy_diff}
\end{align}
cf. Corollary \ref{c:perturbation} and Propositions \ref{p:R_q+1} and \ref{p:energy}.

As in previous papers  \cite{DlSzJEMS,Is2013,BuDLeSz2013,BuDLeIsSz2015,BuDLeSz2016} the key idea, introduced in \cite{DlSz2013}, for reducing the size of the error is to add a highly oscillatory perturbation $w_{q+1}$ to $\overline v_q$. Previous schemes heavily relied on Beltrami flows, but these seemed insufficient to push the method beyond H\"older exponent $\sfrac15$. A new set of flows, called Mikado flows, with much better properties were introduced in \cite{DaSz2016} and indeed, a key element in the proof of Isett \cite{Is2016} is the observation, already used in \cite{DaSz2016}, that Mikado flows behave better under advection by a mean flow. 

An important point is that the Mikado flows will not only be used to ``cancel'' the error $\mathring{\overline{R}}_q$, but also to ``improve the energy'' in areas where the error vanishes identically. In particular, the perturbation will be added in spacetime regions which are disjoint and contained in time-slabs of thickness $2\tau_q$, but with the property that their projections on the time axis is a covering of the interval $[0,T]$. 

\begin{proof}[Proof of Proposition \ref{p:main}] The estimate  \eqref{e:v_diff_prop_est} is a consequence of \eqref{e:v:ell:0}, \eqref{e:v:ell:k}, \eqref{e:overline_v}, \eqref{e:overline_v_N} and \eqref{e:outline_v_diff}:
\[
\norm{v_{q+1}-v_q}_0 + \lambda_{q+1}^{-1} \norm{v_{q+1}-v_1}_1 \leq  \frac{M}{2} \delta_{q+1}^{\sfrac{1}{2}} \lambda_{q+1} + C \delta_{q+1}^{\sfrac{1}{2}} \ell^\alpha + C \delta_q^{\sfrac{1}{2}} \lambda_q \lambda_{q+1}^{-1}\, ,
\]
where the constant $C$ depends on $\alpha, \beta, M$, but not on $a,b$ and $q$. In particular, for every fixed $b$ \eqref{e:v_diff_prop_est} holds if $a$ is large enough. For \eqref{e:v_q_inductive_est}, we use 
the induction assumption to get
\[
\|v_{q+1}\|_1 \leq M \delta_q^{\sfrac{1}{2}} \lambda_q + \frac{M}{2} \delta_{q+1}^{\sfrac{1}{2}} \lambda_{q+1} + C \delta_{q+1}^{\sfrac{1}{2}} \ell^{\alpha} + C \delta_q^{\sfrac{1}{2}} \lambda_q \lambda_{q+1}^{-1}
\]
and again a sufficiently large choice of $a$ will guarantee $\|v\|_{q+1}\leq M \delta_{q+1}^{\sfrac{1}{2}} \lambda_{q+1}$. Similarly for \eqref{e:v_q_0}, which will follow from
\[
\|v_{q+1}\|_0 \leq \|v_q\|_0 + \|v_{q+1}-v_q\|_0 \leq 1-\delta_q^{\sfrac{1}{2}} + M \delta_{q+1}^{\sfrac{1}{2}}\, .
\]

From \eqref{e:outline_R_est} and \eqref{e:outline_energy_diff}, the inequalities \eqref{e:R_q_inductive_est} and \eqref{e:energy_inductive_assumption} follow as a consequence of the parameter inequality
\begin{equation}\label{e:nash_ineq}
\frac{\delta_q^{\sfrac12}\delta_{q+1}^{\sfrac12}\lambda_q}{\lambda_{q+1}}\leq \frac{\delta_{q+2}}{\lambda_{q+1}^{8\alpha}}\,.
\end{equation}
To see this, one divides by the right hand side, takes logarithms and divides by $\log \lambda_{q}$, to obtain
\[-\beta-\beta b+1-b+2b^2\beta+8b\alpha +O\left(\frac{1}{\log \lambda_q}\right)\leq 0, \]
where the error term $O\left(\frac{1}{\log \lambda_q}\right)$ is due to the constants in \eqref{e:bloody_integers}. From the relation \eqref{e:b_beta_rel}, if $\alpha$ is sufficiently small we obtain
\begin{equation}\label{e:choiceofb}
-\beta-\beta b+1-b+2b^2\beta+8b\alpha <0\,. 
\end{equation}
Hence fixing $b$ to satisfy \eqref{e:b_beta_rel}, choosing subsequently $\alpha$ sufficiently small and then $a$ sufficiently large, we obtain \eqref{e:nash_ineq}.

Finally, an entirely analogous argument shows \eqref{e:energy_inductive_assumption} from \eqref{e:outline_energy_diff}.
\end{proof}

\section{Stability estimates for classical exact solutions}\label{s:stability}

\subsection{Classical solutions} For each $i$, let $t_i= i \tau_q$, and consider smooth solutions of the Euler equations
\begin{equation}
\left\{\begin{array}{l}
\partial_t v_i + \div (v_i \otimes v_i) + \nabla p_i =0\\ \\
\div v_i = 0\\ \\
v_i(\cdot,t_i)=v_{\ell}(\cdot, t_i)
\, .
\end{array}\right.
\label{e:vi:def}
\end{equation} 
defined over their own maximal interval of existence.
Next, recall the following

\begin{proposition}
\label{p:local:Euler}
For any $\alpha>0$ there exists a constant $c=c(\alpha)>0$ with the following property. Given any initial data $u_0\in C^{\infty}$, and $T\leq c\norm{u_0}_{1+\alpha}$, there exists a unique solution $u:\mathbb R^3 \times [-T,T]\rightarrow \mathbb R^3$  to the Euler equation
\begin{equation*}
\left\{\begin{array}{l}
\partial_t u+ \div (u\otimes u) + \nabla p = 0\\ \\
\div u = 0\, , \\ \\
u(\cdot, 0)=u_0
\end{array}\right.
\end{equation*} 
Moreover, $u$ obeys the bounds
\begin{align}
\norm{u}_{N+\alpha} \lesssim &\norm{u_0}_{N+\alpha}~. \label{e:euler_eq_bd_k}
\end{align}
for all $N\geq 1$, where the implicit constant depends on $N$ and $\alpha>0$.
\end{proposition}

\begin{proof}[Proof of Proposition~\ref{p:local:Euler}]
The proof of the existence of a unique solution is standard (see e.g.~\cite[Chapter 4]{MaBe2002}), and follows from the restriction $T \leq c \norm{u_0}_{1+\alpha}$. The higher-order bounds \eqref{e:euler_eq_bd_k} are also standard, and can be obtained as follows: For any multi-index $\theta$ with $|\theta|=N$ we have
\[
\partial_t\partial^{\theta}v+v\cdot\nabla\partial^\theta v+[\partial^\theta,v\cdot\nabla]v+\nabla\partial^\theta p=0.
\]
Using the equation for the pressure $-\Delta p=\nabla v\cdot  \nabla v$ and Schauder estimates  we obtain
\[
\|\nabla\partial^{\theta}p\|_{\alpha}\lesssim \|\nabla v\cdot \nabla v\|_{N-1+\alpha}\lesssim \|v\|_{1+\alpha}\|v\|_{N+\alpha}.
\]
Therefore 
\[
\|(\partial_t+v\cdot\nabla)\partial^\theta v\|_{\alpha}\lesssim  \|v\|_{1+\alpha}\|v\|_{N+\alpha},
\]
and \eqref{e:euler_eq_bd_k} follows by applying  \eqref{e:trans_est_alpha}  and Gr\"onwall's inequality.
\end{proof}

An immediate consequence is:
\begin{corollary}
\label{c:size:vi}
If $a$ is sufficiently large, for $\abs{t-t_i}\leq \tau_q$,  we have
\begin{equation}\label{e:z:N+alpha}
\norm{v_i}_{N+\alpha} \lesssim \delta_q^{\sfrac12}\lambda_q\ell^{1-N-\alpha}\lesssim \tau_q^{-1}\ell^{1-N+\alpha}\,\qquad
\mbox{for any $N\geq 1$}.
\end{equation}
\end{corollary}

\begin{proof}[Proof of Corollary~\ref{c:size:vi}]
We apply Proposition \ref{p:local:Euler} and use the estimate \eqref{e:CFL} to obtain
\[
\|v_i\|_{N+\alpha}\lesssim \|v(t_i)\|_{N+\alpha}
\]
for any $N\geq 1$. From \eqref{e:v:ell:k} we then deduce 
the estimate \eqref{e:z:N+alpha}.
\end{proof}

\subsection{Stability and estimates on \texorpdfstring{$v_i-v_\ell$}{vi-vl}}
We will now show that for $\abs{t_i-t}\leq \tau_q$, $v_i$ is close to $v_{\ell}$ and by the identity 
\begin{equation*}
v_i - v_{i+1} = (v_{i}-v_{\ell})-(v_{i+1}-v_{\ell}),
\end{equation*}
the vector field $v_i$ is also close to $v_{i+1}$.

\begin{proposition}
\label{p:vi:vell}
For $\abs{t-t_i}\leq \tau_q$ and $N\geq 0$ we have
\begin{align}
\norm{v_i-v_{\ell}}_{N+\alpha} \lesssim & \tau_q\delta_{q+1}\ell^{-N-1+\alpha}\,, \label{e:z_diff_k}\\
\norm{\nabla(p_{\ell} - p_i)}_{N+\alpha} &\lesssim \delta_{q+1}\ell^{-N-1+\alpha}\,, \label{e:pressure_1}\\
\norm{D_{t,\ell}(v_i-v_{\ell})}_{N+\alpha} &\lesssim  \delta_{q+1}\ell^{-N-1+\alpha}\,, \label{e:z_D_t}
\end{align}
where we write
\begin{align}
D_{t,\ell}=\partial_t+v_{\ell}\cdot\nabla
\label{e:Dtell:def}
\end{align}
for the transport derivative.
\end{proposition}

\begin{proof}[Proof of Proposition~\ref{p:vi:vell}]
Let us first consider \eqref{e:z_diff_k} with $N=0$.
From \eqref{e:euler_reynolds_l} and \eqref{e:vi:def} we have
\begin{equation}\label{e:z_diff_evo}
\partial_t (v_{\ell} - v_i) + (v_{\ell} \cdot \nabla) (v_{\ell} - v_i)
= (v_i - v_{\ell}) \cdot \nabla v_i - \nabla (p_{\ell} - p_i)+\div \mathring{R}_{\ell}.
\end{equation}
In particular, using
\begin{equation}\label{e:eqnpi}
\Delta (p_{\ell} - p_i) = \div\bigl(\nabla v_\ell(v_\ell-v_i)\bigr)+\div\bigl(\nabla v_i(v_\ell-v_i)\bigr)+\div\div\mathring{R}_{\ell},
\end{equation}
estimates \eqref{e:R:ell} and \eqref{e:z:N+alpha}, and Proposition~\ref{p:CZO_C_alpha} (recall that $\partial_{i} \partial_j (-\Delta)^{-1}$ is given by $\sfrac 13 \delta_{ij} ~+ $ a Calder{\'o}n-Zygmund operator), we conclude
\begin{equation*}
\norm{\nabla (p_{\ell} - p_i) (\cdot, t)}_{\alpha} \leq  \delta_q^{\sfrac12}\lambda_q\ell^{-\alpha}\norm{v_{i} - v_\ell}_{\alpha}+\delta_{q+1}\ell^{-1+\alpha}\,.
\end{equation*}
Thus, using \eqref{e:R:ell} and the definition of $\tau_q$, we have
\begin{equation}\label{e:z_D_t_step}
\norm{D_{t,\ell} (v_{\ell} - v_i) }_{\alpha}\lesssim \delta_{q+1}\ell^{-1+\alpha}+\tau_q^{-1} \norm{v_{\ell} - v_i}_{\alpha}
\end{equation}
By applying  \eqref{e:trans_est_alpha} we obtain
\begin{align*}
\norm{ (v_{\ell} - v_i) (\cdot, t) }_{\alpha} \lesssim  \abs{t-t_i}\delta_{q+1}\ell^{-1+\alpha}+ \int_{t_i}^t\tau_q^{-1} \norm{(v_{\ell} - v_i)(\cdot,s)}_{\alpha} ~ds.
\end{align*}
Applying Gr\"onwall's inequality and using the assumption $\abs{t-t_i}\leq \tau_q$  we obtain 
\begin{equation}\label{e:zdiff:N=0}
\norm{v_i-v_{\ell}}_{\alpha} \lesssim \tau_q\delta_{q+1}\ell^{-1+\alpha}\,,
\end{equation}
i.e. \eqref{e:z_diff_k} for the case $N=0$. Then, as a consequence of \eqref{e:z_D_t_step} we obtain \eqref{e:z_D_t} for the case $N=0$.

Next, consider the case $N\geq 1$ and let $\theta$ be a multiindex with $|\theta|=N$. 
Commuting the derivative $\partial^\theta$ with the material derivative $\partial_t + v_{\ell} \cdot \nabla$ we have
\begin{align*}
\|D_{t,\ell}\partial^\theta (v_\ell-v_i)\|_{\alpha}&\lesssim \|\partial^\theta D_{t,\ell} (v_\ell-v_i)\|_{\alpha}+\|[v_\ell\cdot\nabla,\partial^\theta](v_\ell-v_i)\|_\alpha\\
&\lesssim  \|\partial^\theta D_{t,\ell} (v_\ell-v_i)\|_{\alpha}+\|v_\ell\|_{N+\alpha}\|v_\ell-v_i\|_{1+\alpha}+\|v_\ell\|_{1+\alpha}\|v_\ell-v_i\|_{N+\alpha}\\
&\lesssim \|\partial^\theta D_{t,\ell} (v_\ell-v_i)\|_{\alpha}+\|v_\ell\|_{N+1+\alpha}\|v_\ell-v_i\|_{\alpha}+\|v_\ell\|_{1+\alpha}\|v_\ell-v_i\|_{N+\alpha}\,,
\end{align*}
where in the last inequality we used the standard interpolation inequalities on H\"older norms, cf. \eqref{e:Holderinterpolation}. On the other hand
differentiating \eqref{e:z_diff_evo} leads to
\begin{align}
\|\partial^\theta D_{t,\ell} (v_\ell-v_i)\|_{\alpha}&\lesssim \|v_\ell-v_i\|_{N+\alpha}\|v_i\|_{1+\alpha}+\|v_\ell-v_i\|_{\alpha}\|v_i\|_{N+1+\alpha}+\|p_\ell-p_i\|_{N+1+\alpha}+\|\mathring{R}_\ell\|_{N+1+\alpha}\notag\\
&\lesssim \tau_q^{-1} \|v_\ell-v_i\|_{N+\alpha}+\delta_{q+1}\ell^{-N-1+\alpha}+\|\nabla(p_\ell-p_i)\|_{N+\alpha}\label{e:dbetaDt}\,,
\end{align}
where we have used \eqref{e:zdiff:N=0}.
Furthermore, from \eqref{e:eqnpi} we also obtain, using Corollary \ref{c:size:vi} and \eqref{e:zdiff:N=0}
\begin{align}
\|\nabla(p_\ell-p_i)\|_{N+\alpha}&\lesssim (\|v_\ell\|_{N+1+\alpha}+\|v_i\|_{N+1+\alpha})\|v_\ell-v_i\|_{\alpha} \notag\\
&\qquad +(\|v_\ell\|_{1+\alpha}+\|v_i\|_{1+\alpha})\|v_\ell-v_i\|_{N+\alpha}+\|\mathring{R}_\ell\|_{N+1+\alpha}\notag\\
&\lesssim \delta_{q+1}\ell^{-N-1+\alpha}+\tau_q^{-1}\|v_\ell-v_i\|_{N+\alpha}\label{e:pressureNN}\,.
\end{align}
Summarizing, for any multiindex $\theta$ with $|\theta|=N$ we obtain
$$
\|D_{t,\ell}\partial^\theta (v_\ell-v_i)\|_{\alpha}\lesssim \delta_{q+1}\ell^{-N-1+\alpha}+\tau_q^{-1}\|v_\ell-v_i\|_{N+\alpha}.
$$
Therefore, invoking once more  \eqref{e:trans_est_alpha} we deduce 
\begin{equation*}
  \|(v_\ell-v_i)(\cdot,t)\|_{N+\alpha}\lesssim \tau_q\delta_{q+1}\ell^{-N-1+\alpha}+\int_{t_i}^t\tau_q^{-1}\|(v_\ell-v_i)(\cdot,s)\|_{N+\alpha}\,ds	,
\end{equation*}
and hence, using Gr\"onwall's inequality and the assumption $\abs{t-t_i}\leq \tau_q$  we obtain \eqref{e:z_diff_k}. From  \eqref{e:pressureNN} and \eqref{e:dbetaDt} we then also conclude \eqref{e:pressure_1} and \eqref{e:z_D_t}.
\end{proof}

\subsection{Estimates on vector potentials}
Define the vector potentials to the solutions $v_i$ as
\begin{equation}\label{e:ziBidef}
z_i=\RRc v_i:=(-\Delta)^{-1}\curl v_i,
\end{equation}
where $\RRc$ is the Biot-Savart operator, so that
\begin{equation}\label{e:Biot_Savart}
\div z_i=0\qquad\textrm{ and }\qquad\curl z_i=v_i.
\end{equation}
Our aim is to obtain estimates for the differences $z_i-z_{i+1}$. The heuristic is as follows: from Proposition \ref{p:vi:vell} we
obtain
\[
\|v_i-v_{i+1}\|_{N+\alpha}\lesssim \tau_q\delta_{q+1}\ell^{-N-1+\alpha}.
\]
Since the characteristic length-scale of the vectorfields $v_i$ is $\ell$ (cf.~Corollary \ref{c:size:vi}), we expect to gain a factor $\ell$ when passing to first order potentials. This is formalized  in Proposition~\ref{p:S_est} below.

\begin{proposition}\label{p:S_est}
For $\abs{t-t_i} \leq \tau_q$, we have that
\begin{align}
\norm{z_i-z_{i+1}}_{N+\alpha} &\lesssim  \tau_q\delta_{q+1}\ell^{-N+\alpha}\,,   \label{e:z_diff} \\
\norm{D_{t,\ell} (z_i-z_{i+1})}_{N+\alpha} &\lesssim  \delta_{q+1}\ell^{-N+\alpha}\,, \label{e:z_diff_Dt}
\end{align}
where $D_{t,\ell}$ is as in \eqref{e:Dtell:def}.
\end{proposition}

\begin{proof}[Proof of Proposition~\ref{p:S_est}]
Set $\tilde z_i := \RRc (v_i-v_{\ell})$ and observ that
$z_{i}-z_{i+1}=\tilde z_i-\tilde z_{i+1}$. Hence, it suffices to estimate $\tilde z_i$ in place of $z_i-z_{i+1}$.

The estimate on $\norm{\nabla \tilde z_i}_{N-1+\alpha}$ for $N\geq 1$ follows directly from \eqref{e:z_diff_k} and the fact that $\nabla \RRc$ is a bounded operator on H\"older spaces:
\begin{align}\label{e:N>=1}
\norm{\nabla \tilde z_i}_{N-1+\alpha}=& \norm{\nabla \RRc (v_i-v_{\ell})}_{N-1+\alpha} \norm{v_i-v_{\ell}}_{N+\alpha}
\lesssim  \tau_q\delta_{q+1}\ell^{-N+\alpha} ~.
\end{align}
Next, observe that 
\begin{equation}\label{e:eqn-vivell}
\partial_t(v_i-v_\ell)+v_\ell\cdot\nabla(v_i-v_\ell)+(v_i-v_\ell)\cdot\nabla v_i+\nabla (p_i-p_\ell)+\div\mathring{R}_\ell=0.	
\end{equation}
Since $v_i-v_\ell=\curl\tilde z_i$ with $\div\tilde z_i=0$, 
we have\footnote{Here we use the notation $[(z\times\nabla)v]^{ij}=\epsilon_{ikl} z^k\partial_lv^j$ for vector fields $z,v$.}
\begin{align*}
 v_\ell\cdot\nabla(v_i-v_{\ell}) &=\curl\bigl((v_\ell\cdot\nabla)\tilde z_i\bigr)+\div\bigl((\tilde z_i\times \nabla)v_\ell\bigr)\\
	((v_i-v_\ell)\cdot\nabla) v_i &=\div\bigl((\tilde z_i\times \nabla)v_i^T\bigr),
\end{align*}
so that we can write \eqref{e:eqn-vivell} as
\begin{equation}\label{e:curlz_i-eqn}
\curl(\partial_t\tilde z_i+(v_\ell\cdot\nabla)\tilde z_i)=-\div\bigl((\tilde z_i\times\nabla)v_\ell+(\tilde z_i\times\nabla)v_i^T\bigr)-\nabla(p_i-p_\ell)-\div\mathring{R}_\ell.	
\end{equation}
Taking the curl of \eqref{e:curlz_i-eqn} the pressure term drops out. Using in addition that $\div\tilde z_i=\div v_i=0$ and the identity 
$\curl\curl =-\Delta+\nabla\div$, we then arrive at
\begin{equation*}
	-\Delta\bigl(\partial_t\tilde z_i+(v_\ell\cdot\nabla)\tilde z_i\bigr)=F,
\end{equation*}
where 
$$
F=-\nabla\div\left((\tilde z_i\cdot \nabla)v_\ell\right)-\curl\div\left((\tilde z_i\times\nabla)v_\ell+(\tilde z_i\times\nabla)v_i^T\right)-\curl\div\mathring{R}_\ell.
$$
Consequently, 
\begin{align}
	\|\partial_t\tilde z_i+(v_\ell\cdot\nabla)\tilde z_i\|_{N+\alpha} \lesssim \; &(\|v_i\|_{N+1+\alpha}+\|v_\ell\|_{N+1+\alpha})\|\tilde z_i\|_\alpha\notag\\
	&+(\|v_i\|_{1+\alpha}+\|v_\ell\|_{1+\alpha})\|\tilde z_i\|_{N+\alpha}+\|\mathring{R}_\ell\|_{N+\alpha}\notag\\
	\lesssim\; & \tau_q^{-1}\|\tilde z_i\|_{N+\alpha}+\tau_q^{-1}\ell^{-N}\|\tilde z_i\|_\alpha+\delta_{q+1}\ell^{-N+\alpha}.
	\label{e:not_yet_commuted}
\end{align}
Setting $N=0$ and using \eqref{e:trans_est_alpha} and Gr\"onwall's inequality we obtain
$$
\|\tilde z_i\|_{\alpha}\lesssim \tau_q \delta_{q+1}\ell^{\alpha}\, ,
$$
which together with \eqref{e:N>=1} gives \eqref{e:z_diff}. 
Using \eqref{e:z_diff} into \eqref{e:not_yet_commuted} we conclude
$$
\|\partial_t\tilde z_i+(v_\ell\cdot\nabla)\tilde z_i\|_{N+\alpha} \lesssim \delta_{q+1}\ell^{-N+\alpha}\, .
$$
Finally commuting the derivatives in the $N+\alpha$-norm with $D_{t,\ell}$ as in the proof of Proposition \ref{p:vi:vell} and
using again \eqref{e:z_diff} we achieve \eqref{e:z_diff_Dt}.
\end{proof}


\section{Gluing procedure}\label{s:gluing}

Now we proceed to glue the solutions $v_i$ together in order to construct $\overline v_q$. The stability estimates above will be used in order to ensure that $\overline v_q$ remains an approximate solution to the Euler equations. 

\subsection{Partition of unity and definition of \texorpdfstring{$\overline{v}_q$}{bar vq}}
Let
$$
t_i=i\tau_q,\qquad I_i=[t_i+\tfrac{1}{3}\tau_q,t_i+\tfrac{2}{3}\tau_i] \cap [0,T] ,\qquad J_i=(t_i-\tfrac{1}{3}\tau_q,t_i+\tfrac{1}{3}\tau_q) \cap [0,T] \,.
$$
Note that $\{I_i,J_i\}_i$ is a decomposition of $[0,T]$ into pairwise disjoint intervals.  
We define a partition of unity $\{\chi_i\}_i$ in time with the following properties:
\begin{itemize}
\item The cut-offs form a partition of unity
\begin{equation}
\sum_i \chi_i \equiv 1
\label{e:chi:partition}
\end{equation}
\item $\supp \chi_i\cap \supp \chi_{i+2}=\emptyset$ and moreover
\begin{equation}\label{e:chi:time:width}
\begin{split}
\supp \chi_i&\subset (t_i-\tfrac 23 \tau_q,t_i+\tfrac 23 \tau_q)\\
\chi_i(t)&=1\quad\textrm{ for }t\in J_i
\end{split}
\end{equation}
\item For any $i$ and $N$ we have
\begin{equation}
\norm{\partial_t^N \chi_i}_0 \lesssim \tau_q^{-N} \label{e:dt:chi}\,.
\end{equation}
\end{itemize}


We define
\begin{align*}
\overline v_q&=\sum_i \chi_i v_i\\
\overline p_q^{(1)}&=\sum_i \chi_i p_i
\end{align*}
Observe that $\div \overline v_q=0$. Furthermore, if $t\in I_i$, then $\chi_i+\chi_{i+1}=1$ and $\chi_j=0$ for $j\neq i,i+1$, therefore
on $I_i$:
\begin{align*}
\overline v_q&=\chi_i v_i+(1-\chi_i) v_{i+1}\\
\overline p_q^{(1)}&=\chi_i p_i+(1-\chi_i) p_{i+1}
\end{align*}
and
\begin{align*}
\partial_t\overline v_q+\div(\overline v_q\otimes \overline v_q)+\nabla\overline p_q^{(1)}
&= \chi_i\partial_t v_i+(1-\chi_i)\partial_t v_{i+1}+\partial_t\chi_i(v_i-v_{i+1})\\
&\quad +\div\left(\chi_i^2v_i\otimes v_i+(1-\chi_i)^2v_{i+1}\otimes v_{i+1}\right)\\
&\quad +\chi_i(1-\chi_i)\div(v_i\otimes v_{i+1}+v_{i+1}\otimes v_{i}))\\
&\quad +\chi_i\nabla p_i+(1-\chi_i)\nabla p_{i+1}\\
&=\partial_t\chi_i(v_i-v_{i+1})-\chi_i(1-\chi_i)\div\left((v_i-v_{i+1})\otimes (v_i-v_{i+1})\right).
\end{align*}
On the other hand, if $t\in J_i$ then $\chi_i=1$ and $\chi_j=0$ for all $j\neq i$ for all $\tilde t$ sufficiently close to $t$ (since $J_i$ is open). Then for all $t\in J_i$ we have
$$
\overline v_q=v_i,\quad \overline p_q^{(1)}=p_i, 
$$
and, from \eqref{e:vi:def}, 
$$
\partial_t\overline v_q+\div(\overline v_q\otimes \overline v_q)+\nabla\overline p_q^{(1)}=0.
$$
\subsection{The new Reynods tensor}
In order to define the new Reynolds tensor, we recall the operator $\mathcal R$ from \cite{DlSz2013}, which
can be thought of as an ``inverse divergence'' operator for symmetric tracefree 2-tensors. The operator is defined as
\begin{equation}
\label{e:R:def}
\begin{split}
({\mathcal R} f)^{ij} &= {\mathcal R}^{ijk} f^k \\
{\mathcal R}^{ijk} &= - \frac 12 \Delta^{-2} \partial_i \partial_j \partial_k + \frac 12 \Delta^{-1} \partial_k \delta_{ij} -  \Delta^{-1} \partial_i \delta_{jk} -  \Delta^{-1} \partial_j \delta_{ik}.
\end{split}
\end{equation}
when acting on vectors $f$ with zero mean on $\T^3$. The following statement, taken from \cite{DlSz2013}, can be proved by direct calculation.

\begin{proposition}\label{p:R}
The tensor ${\mathcal R}$ defined in \eqref{e:R:def} is symmetric, and  we have
\[
\div ( {\mathcal R}  f) = f
\]
for any $f$ with zero mean on $\T^3$.
\end{proposition}

We define
\begin{align*}
\mathring{\overline{R}}_q&=\partial_t\chi_i\mathcal{R}(v_i-v_{i+1})-\chi_i(1-\chi_i)(v_i-v_{i+1})\mathring{\otimes} (v_i-v_{i+1})\\
\overline{p}_q^{(2)}&=-\chi_i(1-\chi_i)|v_i-v_{i+1}|^2,
\end{align*}
for $t\in I_i$ and $\mathring{\overline{R}}_q=0$, $\overline{p}_q^{(2)}=0$ for $t\notin\bigcup_{i}I_i$. Furthermore, we set
$$
\overline{p}_q=\overline{p}_q^{(1)}+\overline{p}_q^{(2)}
$$
It follows from the preceding discussion and Proposition \ref{p:R} that
\begin{itemize}
\item $\mathring{\overline{R}}_q$ is a smooth symmetric and traceless 2-tensor;
\item For all $(x,t)\in \T^3\times [0,T]$
\begin{equation*}
\left\{\begin{array}{l}
\partial_t\overline{v}_q+\div(\overline{v}_q\otimes\overline{v}_q)+\nabla \overline{p}_q =\div \mathring{\overline{R}}_q,\\ \\
\div \overline{v}_q =0;
\end{array}\right.
\end{equation*}
\item $\supp\mathring{\overline{R}}_q\subset \T^3\times \bigcup_iI_i$.
\end{itemize}

\subsection{Estimates on \texorpdfstring{$\overline{v}_q$}{overline vq}}
Next, we estimate the various H\"older norms of $\overline{v}_q$ and $\mathring{\overline{R}}_q$ in order to obtain 
\eqref{e:overline_v}-\eqref{e:Dt:overline_R}. 

\begin{proposition}\label{p:vq:vell}
The velocity field $\overline v_q$ satisfies the following estimates
\begin{align}
\norm{\bar v_q - v_{\ell}}_{\alpha} &\lesssim \delta_{q+1}^{\sfrac12}\ell^{\alpha} \label{e:vq:vell} \\
\norm{\overline{v}_q-v_\ell}_{N+\alpha} &\lesssim \tau_q\delta_{q+1}\ell^{-1-N+\alpha} \label{e:vq:vell:additional} \\
\norm{\bar v_q}_{1+N} &\lesssim \delta_{q}^{\sfrac12} \lambda_q \ell^{-N}\label{e:vq:1} 
\end{align}
for all $N \geq 0$.
\end{proposition}
In particular, this lemma shows that the claimed estimates \eqref{e:overline_v}--\eqref{e:overline_v_N} indeed hold.

\begin{proof}[Proof of Proposition~\ref{p:vq:vell}]
By definition
$$
\overline{v}_q-v_\ell=\sum_i\chi_i(v_i-v_\ell).
$$
Therefore Proposition \ref{p:vi:vell} implies
\begin{equation}\label{e:vq:vell:additional1}
\|\overline{v}_q-v_\ell\|_{N+\alpha}\lesssim \tau_q\delta_{q+1}\ell^{-1-N+\alpha}.
\end{equation}
Note that using the definition of $\ell$ in \eqref{e:ell_def} and $\tau_q$ in \eqref{e:tau_def} and the comparison \eqref{e:compare-lambda-ell} 
\begin{equation}\label{e:calculation11}
\delta_{q+1}^{\sfrac{1}{2}} \tau_q\ell^{-1}=\ell^{2\alpha} \lambda_q^{\sfrac{3\alpha}{2}}\leq \lambda_q^{-\sfrac{\alpha}{2}}\leq 1\, .
\end{equation}
Therefore we obtain \eqref{e:vq:vell}, and furthermore, for any $N\geq 0$
\begin{align*}
\|\overline{v}_q-v_\ell\|_{1+N+\alpha}&\lesssim \delta_{q+1}\tau_q\ell^{-N-2+\alpha}= \delta_q^{\sfrac12}\lambda_q(\ell\lambda_q)^{3\alpha} \ell^{-N} \leq \delta_q^{\sfrac12}\lambda_q\ell^{-N}.
\end{align*}
Then it also follows using \eqref{e:v:ell:k} that
\begin{align*}
\|\overline{v}_q\|_{1+N}\lesssim& \|v_\ell\|_{1+N}+\|v_\ell-\overline{v}_q\|_{1+N+\alpha}
\lesssim  \delta_q^{\sfrac12}\lambda_q\ell^{-N}.\qedhere
\end{align*}
\end{proof}

\subsection{Estimates on the stress tensor}
We are now in a position to estimate the glued stress tensor $\mathring{\overline R_q}$:
\begin{proposition}\label{p:Rq}
The stress tensor $\mathring{\overline R_q}$ satisfies the following bounds for any $N\geq 0$:
\begin{align} 
\norm{\mathring{\overline R_q}}_{N+\alpha} &\lesssim \delta_{q+1}\ell^{-N+\alpha} \label{e:Rq:1}\\
\norm{(\partial_t + \overline v_q\cdot \nabla) \mathring{\overline R_q}}_{N+\alpha} &\lesssim \delta_{q+1}\delta_q^{\sfrac12}\lambda_q\ell^{-N-\alpha}. \label{e:Rq:Dt}
\end{align}
\end{proposition}
This shows that the claimed estimates \eqref{e:overline_R}--\eqref{e:Dt:overline_R} are indeed obeyed by $\mathring{\overline R_q}$.

\begin{proof}[Proof of Proposition~\ref{p:Rq}]
Recall that $v_i=\curl z_i$, so that we may write for $t\in I_i$:
$$
\mathring{\overline{R}}_q=\partial_t\chi_i(\mathcal{R}\curl)(z_i-z_{i+1})-\chi_i(1-\chi_i)(v_i-v_{i+1})\mathring\otimes (v_i-v_{i+1}).
$$
Note that $\mathcal{R}\curl$ is a zero-order operator. Therefore we obtain from Propositions \ref{p:vi:vell} and \ref{p:S_est} for any $N\geq 0$ with $t\in I_i$
\begin{align*}
\|\mathring{\overline{R}}_q\|_{N+\alpha}&\lesssim \tau_q^{-1}\|z_i-z_{i+1}\|_{N+\alpha}+\|v_i-v_{i+1}\|_{N+\alpha}\|v_i-v_{i+1}\|_\alpha\\
&\lesssim \delta_{q+1}\ell^{-N+\alpha}+\tau_q^2\delta_{q+1}^2\ell^{-2-N+2\alpha} \lesssim \delta_{q+1}\ell^{-N+\alpha}.
\end{align*}
Here we used again \eqref{e:calculation11}.
Next, we calculate
\begin{align*}
D_{t,\ell}\mathring{\overline{R}}_q
&=\partial_t^2\chi_i(\mathcal{R}\curl)(z_i-z_{i+1})\\
&\quad +\partial_t\chi_i(\mathcal{R}\curl)D_{t,\ell}(z_i-z_{i+1})+\partial\chi_i[v\cdot\nabla,\mathcal{R}\curl](z_i-z_{i+1})\\
&\quad -\partial_t(\chi_i(1-\chi_i))(v_i-v_{i+1})\mathring\otimes (v_i-v_{i+1})\\
&\quad -\chi_i(1-\chi_i))\Bigl((D_{t,\ell}(v_i-v_{i+1}))\mathring\otimes (v_i-v_{i+1})-(v_i-v_{i+1})\mathring\otimes (D_{t,\ell}(v_i-v_{i+1}))\Bigr),
\end{align*}
where $[v\cdot\nabla,\mathcal{R}\curl]$ denotes the commutator. 
Hence, using Proposition \ref{p:com:CZ:multiplication} and Propositions  \ref{p:vi:vell} and \ref{p:S_est} we deduce
\begin{align*}
\|D_{t,\ell}\mathring{\overline{R}}_q\|_{N+\alpha}
&\lesssim \tau_q^{-2}\|z_i-z_{i+1}\|_{N+\alpha}+\tau_q^{-1}\|D_{t,\ell}(z_i-z_{i+1})\|_{N+\alpha} \\
&\quad +\tau_q^{-1}\|v_\ell\|_\alpha\|z_i-z_{i+1}\|_{N+\alpha}+\tau_q^{-1}\|v_\ell\|_{N+\alpha}\|z_i-z_{i+1}\|_{\alpha} \\
&\quad +\tau_q^{-1}\|v_i-v_{i+1}\|_{N+\alpha}\|v_i-v_{i+1}\|_\alpha \\
&\quad +\|D_{t,\ell}(v_i-v_{i+1})\|_{N+\alpha}\|v_i-v_{i+1}\|_\alpha+\|v_i-v_{i+1}\|_{N+\alpha}\|D_{t,\ell}(v_i-v_{i+1})\|_\alpha\\
&\lesssim \tau_q^{-1}\delta_{q+1}\ell^{-N+\alpha}+(\tau_q^2\delta_{q+1}\ell^{-2})\tau_q^{-1}\delta_{q+1}\ell^{-N+2\alpha}\\
&\lesssim \tau_q^{-1}\delta_{q+1}\ell^{-N+\alpha}\,.
\end{align*}
Finally, we deduce using \eqref{e:vq:vell:additional}:
\begin{align*}
\norm{(\partial_t + \overline v_q\cdot \nabla) \mathring{\overline R_q}}_{N+\alpha}&\lesssim \|(v_\ell-\overline{v}_q)\cdot\nabla \mathring{\overline R_q}\|_{N+\alpha}+\|D_{t,\ell}\mathring{\overline{R}}_q\|_{N+\alpha}\\
&\lesssim \|v_\ell-\overline{v}_q\|_{N+\alpha}\|\mathring{\overline{R}}_q\|_{1+\alpha}  +\|v_\ell-\overline{v}_q\|_\alpha\|\mathring{\overline{R}}_q\|_{N+1+\alpha}+\|D_{t,\ell}\mathring{\overline{R}}_q\|_{N+\alpha}\\
&\lesssim  \tau_q\delta_{q+1}^2\ell^{-N-2+2\alpha} +\tau_q^{-1}\delta_{q+1}\ell^{-N+\alpha}\\
&\lesssim \tau_q^{-1}\delta_{q+1}\ell^{-N+\alpha} =\delta_{q+1}^{\sfrac12}\delta_q^{\sfrac12}\lambda_q\ell^{-N-\alpha}
\end{align*}
again using \eqref{e:calculation11}.
\end{proof}

To finish this section we show that $\overline v_q$ has approximately the same energy as $v_\ell$:
\begin{proposition}
The difference of the energies of $\overline v_q$ and $v_\ell$ satisfies
\label{p:glued:energy}
\begin{align}
\left| \int_{\T^3} |\bar v_q|^2 - |v_\ell |^2 dx \right| \lesssim \delta_{q+1}\ell^\alpha\label{e:voverline_vell_energy_diff}
\end{align} 
\end{proposition}
\begin{proof}[Proof of Proposition~\ref{p:glued:energy}]
Observe that for $t\in I_i$
\begin{align*}
 \overline v_q \otimes \overline v_q&=(\chi_iv_i+(1-\chi_i)v_{i+1})\otimes(\chi_iv_i+(1-\chi_i)v_{i+1})\\
 &=\chi_iv_i\otimes v_i+(1-\chi_i)v_{i+1}\otimes v_{i+1}-\chi_i(1-\chi_i)(v_i-v_{i+1})\otimes (v_i-v_{i+1}),
 \end{align*}
 so that, taking the trace:
 \begin{align*}
|\overline v_q|^2-|v_\ell|^2=\chi_i(|v_i|^2-|v_\ell|^2)+(1-\chi_i)(|v_{i+1}|^2-|v_\ell|^2)-\chi_i(1-\chi_i)|v_i-v_{i+1}|^2
\end{align*}
Next, recall that $v_i$ and $v_\ell$ are smooth solutions of \eqref{e:vi:def} and \eqref{e:euler_reynolds_l} respectively, therefore
\begin{align*}
\left|\frac{d}{dt}\int_{\T^3}|v_i|^2-|v_\ell|^2\,dx\right|=\left|\int_{\T^3}\nabla v_\ell:\mathring{R}_\ell\,dx\right|
&\lesssim \|\nabla v_\ell\|_0 \|\mathring{R}_\ell\|_0 \\
&\lesssim \delta_q^{\sfrac{1}{2}} \lambda_q \delta_{q+1} \lesssim\tau_q^{-1}\delta_{q+1}\ell^\alpha,
\end{align*}
where we have used \eqref{e:R:ell} and \eqref{e:z:N+alpha}. Moreover, $v_i=v_\ell$ for $t=t_i$. Therefore, after integrating in time we deduce
$$
\left|\int_{\T^3}|v_i|^2-|v_\ell|^2\,dx\right|\lesssim \delta_{q+1}\ell^\alpha.
$$
Furthermore, using \eqref{e:z_diff_k} and \eqref{e:calculation11} 
$$
\int_{\T^3}|v_i-v_{i+1}|^2\,dx\lesssim \|v_i-v_{i+1}\|_\alpha^2\lesssim \tau_q^2\delta_{q+1}^2\ell^{-2+2\alpha} \stackrel{\eqref{e:calculation11}}{\lesssim} \delta_{q+1}\ell^{2\alpha},
$$
Therefore 
$$
\left| \int |\bar v_q|^2 - |v_\ell |^2 dx \right|\lesssim \delta_{q+1}\ell^{\alpha},
$$
which concludes the proof.
\end{proof}

\section{Perturbation step}\label{s:perturbation}

In this section, we will outline the construction of the perturbation $w_{q+1}$, where 
\[
v_{q+1}:= w_{q+1} + \overline v_q \, ,
\] 
As already explained in the outline of the proof, the perturbation $w_{q+1}$ is highly oscillatory and will be based on the Mikado flows introduced in \cite{DaSz2016}, which are designed to cancel the low frequency error $\overline  R_q$ and are Lie-advected by the mean flow of $\overline v_q$.

\subsection{Mikado flows}
We begin by recalling the construction of Mikado flows given in \cite{DaSz2016}.

\begin{lemma}\label{l:Mikado}For any compact subset $\mathcal N\subset\subset \S^{3\times3}_+$
there exists a smooth vector field 
$$
W:\mathcal N\times \T^3 \to \R^3, 
$$
such that, for every $R\in\mathcal N$ 
\begin{equation}\label{e:Mikado}
\left\{\begin{aligned}
\div_\xi(W(R,\xi)\otimes W(R,\xi))&=0 \\ \\
\div_\xi W(R,\xi)&=0,
\end{aligned}\right.
\end{equation}
and
\begin{eqnarray}
	\fint_{\T^3} W(R,\xi)\,d\xi&=&0,\label{e:MikadoW}\\
    \fint_{\T^3} W(R,\xi)\otimes W(R,\xi)\,d\xi&=&R.\label{e:MikadoWW}
\end{eqnarray}
\end{lemma}

Using the fact that $W(R,\xi)$ is $\T^3$-periodic and has zero mean in $\xi$, we write
\begin{equation}\label{e:Mikado_Fourier}
W(R,\xi)=\sum_{k \in \Z^3\setminus\{0\}} a_k(R)A_k e^{ik\cdot \xi}
\end{equation}
for some coefficients $a_k(R)$ and complex vector $A_k\in \C^3$, satisfying $A_k\cdot k=0$ and $\abs{A_k}=1$. From the smoothness of $W$, we further infer
\begin{equation}\label{e:a_k_est}
\sup_{R\in \mathcal N}\abs{D^N_R a_k(R)}\leq  \frac{C(\mathcal{N},N,m)}{\abs {k}^m}
\end{equation}
for some constant $C$, which depends, as highlighted in the statement, on $\mathcal{N}$, $N$ and $m$.
\begin{remark}\label{r:choice_of_M}
Later in the proof the estimates \eqref{e:a_k_est} will be used with a specific choice of the compact set $\mathcal{N}$ and of the integers $N$ and $m$: this specific choice will then determine the universal constant $M$ appearing in Proposition \ref{p:main}.
\end{remark}

Using the Fourier representation we see that from \eqref{e:MikadoWW}
\begin{equation}\label{e:Mikado_stationarity}
W(R,\xi)\otimes W(R,\xi) = R+\sum_{k\neq 0} C_{k}(R) e^{i k\cdot \xi}
\end{equation}
where
\begin{equation}\label{e:Ck_ind}
C_k  k=0 \quad \mbox{and} \quad
\sup_{R\in \mathcal N}\abs{D^N_R C_k(R)}\leq \frac{C (\mathcal{N}, N, m)}{\abs {k}^m}
\end{equation}
for any $m,N \in \N$. 

It will also be useful to write the Mikado flows in terms of a potential. We note
\begin{align}
\curl_{\xi}\left(\left(\frac{ik\times A_k}{\abs{k}^2}\right) e^{k\cdot \xi}\right) &= -i\left(\frac{ik\times A_k}{\abs{k}^2}\right)\times k  e^{k\cdot \xi} 
= -\frac{k\times (k\times A_k)}{\abs{k}^2}  e^{k\cdot \xi} = A_k  e^{k\cdot \xi} \label{e:Mikado_Potential}
\end{align}

\subsection{Squiggling stripes and the stress tensor \texorpdfstring{$\tilde{R}_{q,i}$}{tilde Rqi}}
Recall that $\mathring{\overline R_q}$ is supported in the set $\T^3\times \bigcup_iI_i$, whereas, from \eqref{e:chi:time:width} it follows that 
$[0,T]\setminus  \bigcup_iI_i=\bigcup_iJ_i$, where the open intervals $J_i$ have length $|J_i|=\tfrac 23\tau_q$ each, except for the first and last one, which might be shortened by the intersection with $[0,T]$, more precisely
$$
J_i=(t_i- \tfrac 13 \tau_q,t_i+\tfrac 13\tau_q) \cap [0,T] \, .
$$
We start by defining smooth non-negative cut-off functions $\eta_i=\eta_i(x,t)$ with the following properties
\begin{enumerate}
\item[(i)] $\eta_i\in C^{\infty}(\T^3\times [0,T])$ with $0\leq \eta_i(x,t)\leq 1$ for all $(x,t)$;
\item[(ii)] $\supp \eta_i\cap\supp\eta_j=\emptyset$ for $i\neq j$;
\item[(iii)] $\T^3\times I_i\subset \{(x,t):\eta_i(x,t)=1\}$;
\item[(iv)] $\supp \eta_i\subset \T^3\times I_i\cup J_i\cup J_{i+1}=\T^3\times (t_i- \tfrac 13 \tau_q,t_{i+1}+\tfrac 13 \tau_q)\cap [0,T]$;
\item[(v)] There exists a positive geometric constant $c_0>0$ such that for any $t\in[0,T]$
$$
\sum_i\int_{\T^3}\eta_i^2(x,t)\,dx\geq c_0.
$$
\end{enumerate}
In view of (iv) we set
$$
\tilde I_i=(t_i- \tfrac 13 \tau_q,t_{i+1}+\tfrac 13 \tau_q ) \cap [0,T] \,.
$$
\begin{lemma}\label{l:cutoffs}
There exists cut-off functions $\{\eta_i\}_i$ with the properties (i)-(v) above and such that for any $i$ and $n,m\geq 0$
\begin{align*}
\|\partial_t^n\eta_i\|_{m}\leq C (n,m) \tau_q^{-n}
\end{align*}
where $C(n,m)$ are geometric constants depending only upon $m$ and $n$.
\end{lemma}

\begin{proof}[Proof of Lemma~\ref{l:cutoffs}]
First of all we consider the sharp cutoffs $\tilde{\eta}_i$ defined by
\begin{align*}
\tilde \eta_i &=  {\mathbf 1}_{\tilde \Omega_i} \\
\tilde \Omega_i &= \left\{ (x,t) \colon t_i + \tfrac{\tau_q}{6} ( \sin(2\pi x_1) + \tfrac 12 ) \leq t \leq t_{i+1} + \tfrac{\tau_q}{6} (\sin(2\pi x_1) - \tfrac 12) \right\}
\end{align*}
Next we fix a standard mollifier $\varkappa$ in time and the standard mollifier $\psi$ in space already used so far.
Hence we define $\eta_i$ by mollifying $\tilde \eta_i$ in space and time as follows:
\begin{equation*}
\eta_i (x,t) = \int \tilde \eta_i (y, s) \psi \left(\frac{x-y}{c_1}\right) \varkappa \left(\frac{t-s}{c_2 \tau_q}\right)\, dy\, ds\, ,
\end{equation*}
 where $c_1$ and $c_2$ are positive geometric constants. One may check that a suitable choice of $c_1$ and $c_2$ yields the desired conclusions (see Figure~\ref{f:squiggles}).\end{proof}

\begin{figure}
\begin{center}
\includegraphics[scale=0.25]{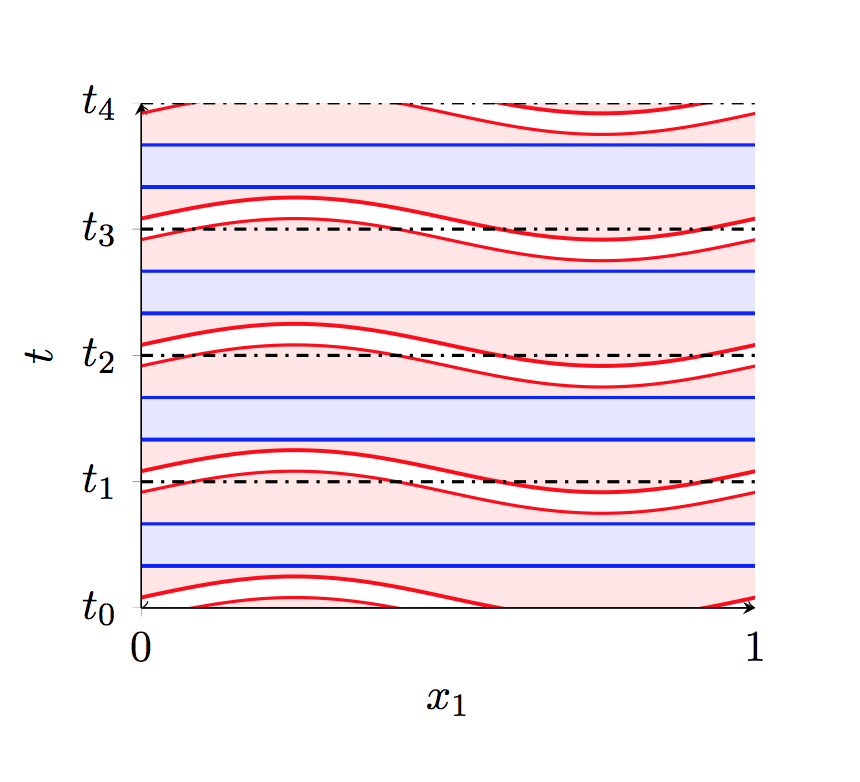}
\end{center}
\caption{The support of $\mathring{\overline R_q}$ is given by the blue regions. The support of the cut-off functions $\eta_i$ are encapsulated in the red region.}\label{f:squiggles}
\end{figure}

Define 
\begin{equation*}
\rho_{q}(t):= \frac{1}{3} \left(e(t)-\frac{\delta_{q+2}}{2}-\int_{\T^3}\abs{\overline v_q}^2\,dx\right)
\end{equation*}
and
\begin{equation*}
\rho_{q,i}(x,t):= \frac{\eta_i^2(x,t)}{\sum_j \int_{\T^3} \eta_j^2(y,t)\,dy}\rho_{q}(t)
\end{equation*}

Define the backward flows $\Phi_i$ for the velocity field $\overline v_{q}$ as the solution of the transport equation
\begin{equation*}
\left\{ 
\begin{aligned}
&(\partial_t + \overline v_q  \cdot \nabla) \Phi_i =0 \\ \\
&\Phi_i\left(x,t_i\right) = x.
\end{aligned}
\right.
\end{equation*}
Define
\begin{equation*}
R_{q,i}:=\rho_{q,i} \Id- \eta_i^2\mathring{\overline R_q}
\end{equation*}
and
\begin{equation}\label{e:tildeR_def}
\tilde R_{q,i} =  \frac{\nabla\Phi_iR_{q,i}(\nabla\Phi_i)^T}{ \rho_{q,i}} \,.
\end{equation}
We note that, because of properties (ii)-(iv) of $\eta_i$, 
\begin{itemize}
\item $\supp R_{q,i}\subset \supp\eta_i$ and on $\supp\eta_i$ we have $R_{q,i}=\rho_{q+1,i} \Id- \mathring{\overline R_q}$;
\item $\supp \tilde R_{q,i}\subset \T^3\times (t_i- \tfrac 13 \tau_q,t_{i+1}+ \tfrac 13 \tau_q)$;
\item $\supp \tilde R_{q,i}\cap \supp \tilde R_{q,j}=\emptyset\textrm{ for all }i\neq j$.
\end{itemize}

\begin{lemma}
\label{l:R_in_range}
For $a\gg 1$ sufficiently large we have
\begin{equation}\label{e:Phi-close-to-id}
\|\nabla \Phi_i - \Id\|_0 \leq \frac{1}{2} \qquad \mbox{for $t\in \supp (\eta_i)$.}
\end{equation}
Furthermore, for any $N\geq 0$ 
\begin{align}
\frac{\delta_{q+1}}{8\lambda_q^{\alpha}} \leq |\rho_{q}(t)| &\leq \delta_{q+1}\quad\textrm{ for all $t$}\,,
\label{e:rho_range}\\
\norm{\rho_{q,i}}_0 &\leq \frac{\delta_{q+1}}{c_0}\,,\label{e:rho_i_bnd}\\
 \norm{\rho_{q,i}}_N&\lesssim \delta_{q+1}\,,\label{e:rho_i_bnd_N}\\
 \norm{\partial_t \rho_q}_0 &\lesssim \delta_{q+1} \delta_q^{\sfrac{1}{2}} \lambda_q
 \label{e:rho_t}\,,\\
 \norm{\partial_t \rho_{q,i}}_N &\lesssim \delta_{q+1}\tau_q^{-1}\,.
 \label{e:rho_i_bnd_t}
\end{align}
Moreover, for all $(x,t)$
$$
\tilde R_{q,i}(x,t)\in B_{\sfrac12}(\Id)\subset \mathcal{S}^{3\times 3}_+\,,
$$
where $B_{\sfrac12}(\Id)$ denotes the metric ball of radius $1/2$ around the identity $\Id$ in the space $\mathcal{S}^{3\times 3}$. 
\end{lemma}
\begin{proof}[Proof of Lemma~\ref{l:R_in_range}]
Note that \eqref{e:rho_range} is a trivial consequence of estimate \eqref{e:voverline_energy_error} and the inequality $4\delta_{q+2}\leq \delta_{q+1}$.  Note that by the definition of the cut-off functions $\eta_i$ 
\begin{equation}\label{e:sum_eta_range}
c_0 \leq \sum_i \int_{\T^3} \eta_i^2(y,t)\,dy \leq 2
\end{equation}
and hence we obtain \eqref{e:rho_i_bnd}. Since $\abs{\nabla^N \eta_j}\lesssim 1$, the bound \eqref{e:rho_i_bnd_N} also follows. 

Next, note that by applying \eqref{e:overline_v_N} and  \eqref{e:Dphi_near_id} we obtain
\begin{equation*}
\norm{\nabla\Phi_i-\Id}_0 \lesssim \tau_q  \delta_q^{\sfrac 12} \lambda_q  = \ell^{2\alpha}.
\end{equation*}
Furthermore, by definition we have
\begin{align*}
\tilde R_{q,i}-\Id&=\nabla\Phi_i\left(\frac{R_{q,i}}{\rho_{q,i}}-\Id\right)\nabla\Phi_i^T+\nabla\Phi_i \nabla\Phi_i^T-\Id\\
&=\nabla\Phi_i\frac{\eta_i^2\mathring{\overline R}_{q}}{\rho_{q,i}}\nabla\Phi_i^T+\nabla\Phi_i\nabla\Phi_i^T-\Id
\end{align*}
Using \eqref{e:overline_R} we see that
$$
\left| \frac{\eta_i^2\mathring{\overline R}_{q}}{\rho_{q,i}}\right|\lesssim \frac{1}{\delta_{q+1}}\left|\mathring{\overline R}_{q}\right|\lesssim \ell^{\alpha}.
$$
Consequently we obtain
$$
|\tilde R_{q,i}-\Id|\lesssim \ell^{\alpha}
$$
so that, choosing $a$ sufficiently large, we ensure that $\tilde R_{q,i}(x,t)$ is contained in the ball of symmetric matrices $B_{\sfrac12}(\Id)$.

Finally, to prove \eqref{e:rho_i_bnd_t} we first note that
\begin{equation*}
\abs{\frac{d}{dt} \int \abs{\overline v_{q}(x,t)}^2\,dx}= \abs{2\int \nabla \overline v_q\cdot \mathring{\overline R_q}\,dx }\lesssim \delta_{q+1}\delta_q^{\sfrac 12}\lambda_q
\end{equation*}
Thus 
\[
 \norm{\partial_t \rho_{q}}_0\lesssim \delta_{q+1}\delta_q^{\sfrac 12}\lambda_q
\]
Then, since $\|\partial_t \eta_j\|_N \lesssim \tau_q^{-1}$ and $\delta_q^{\sfrac 12}\lambda_q\leq \tau_q^{-1}$, 
using \eqref{e:sum_eta_range}, the estimate \eqref{e:rho_i_bnd_t} follows.
\end{proof}

\subsection{The perturbation and the constant \texorpdfstring{$M$}{M}}
The principal term of the perturbation can be written as
\begin{equation}
w_{o}:=\sum_i  \left(\rho_{q,i}(x,t)\right)^{\sfrac12} (\nabla\Phi_i)^{-1} W(\tilde R_{q,i}, \lambda_{q+1}\Phi_i) = \sum_i w_{o,i}\, ,
\label{e:w0_decomp}
\end{equation}
where Lemma \ref{l:Mikado} is applied with $\mathcal{N} = \overline{B}_{\sfrac12} (\Id)$, namely the closed ball (in the space of symmetric $3\times 3$ matrices) of radius $\sfrac{1}{2}$ centered at the identity matrix.

From Lemma \ref{l:R_in_range} it follows that $W(\tilde R_{q,i}, \lambda_{q+1}\Phi_i)$ is well defined. Using the Fourier series representation of the Mikado flows \eqref{e:Mikado_Fourier} we obtain
\begin{equation*}
w_{o,i}:=\sum_{k\neq 0} \left(\rho_{q,i}(x,t)\right)^{\sfrac12} a_k(\tilde R_{q,i})(\nabla\Phi_i)^{-1} A_k e^{i\lambda_{q+1}k\cdot \Phi_i}. 
\end{equation*}
The choice of $w_{o}$ is motivated by the fact that  the vector fields 
\begin{equation*}
U_{i,k}:=(\nabla\Phi_i)^{-1} A_k e^{i\lambda_{q+1}k\cdot \Phi_i}
\end{equation*}
 are Lie-advected by the flow $\overline v_q$:
\begin{equation}\label{e:Lie_advection}
\left(\partial_t +\overline v_q\cdot \nabla\right) U_{i,k}= (D \overline v_q)^T U_{i,k}\,,
\end{equation} 
and thus remain divergence free. For notational convenience we set
\begin{equation*}
b_{i,k}(x,t):= \left(\rho_{q,i}(x,t)\right)^{\sfrac12} a_k(\tilde R_{q,i}(x,t))A_k
\end{equation*}
so that we may write
\begin{equation*}
w_{o,i}:=\sum_{k \neq 0} (\nabla\Phi_i)^{-1} b_{i,k} e^{i\lambda_{q+1}k\cdot \Phi_i}. 
\end{equation*}

The following is a crucial point of our construction, which ensures that the constant $M$ of Proposition \ref{p:main}
is geometric and in particular independent of all the parameters of the construction.

\begin{lemma}\label{l:choice_of_M}
There is a geometric constant $\bar M$ such that
\begin{equation}\label{e:barM}
\|b_{i,k}\|_0 \leq \frac{\bar M}{|k|^5} \delta_{q+1}^{\sfrac{1}{2}}\, .
\end{equation}
\end{lemma}
\begin{proof}[Proof of Lemma~\ref{l:choice_of_M}]
First of all, applying \eqref{e:a_k_est} with $N=0, m=5$ and $\mathcal{N} = \overline{B}_{\sfrac{1}{2}} (\Id)$, we achieve
\[
\|a_{k} (\tilde{R}_{q,i})\|_0 \leq \frac{\bar{C}}{|k|^4}\, ,
\]
where $\bar C$ is a geometric constant (cf. Remark \ref{r:choice_of_M}). Hence, considering the bound \eqref{e:rho_i_bnd}, the constant $\bar M$ is given by $\bar C c_0^{-\sfrac{1}{2}}$.
\end{proof}

We are finally ready to define the constant $M$ of Proposition \ref{p:main}: from Lemma \ref{l:choice_of_M} it follows trivially that the constant is indeed geometric and hence independent of all the parameters entering in the statement of Proposition \ref{p:main}.

\begin{definition}\label{d:choice_of_M}
The constant $M$ is defined as
\[
M = 64 \bar M \sum_{k\in \Z^3\setminus \{0\}} \frac{1}{|k|^4}\, ,
\]
where $\bar M$ is the constant of Lemma \ref{l:choice_of_M}.
\end{definition}

In order to ensure $w_{q+1}$ is divergence free, we correct our principal perturbation $w_o$ by $w_c$, i.e.\ $w_{q+1}=w_o+w_c$ so that $w_{q+1}$ is the curl of  a vector field. In particular, in view of the identity \eqref{e:Mikado_Potential} we define
\begin{align*}
w_{c}&:=\frac{-i}{\lambda_{q+1}}\sum_{i,k\neq 0} \nabla( \left(\rho_{q,i}\right)^{\sfrac12} a_{k}(\tilde R_{q,i}))\times \frac{\nabla\Phi_i^T(k\times A_k)}{\abs{k}^2} e^{i\lambda_{q+1}k\cdot \Phi_i}
=\sum_{i,k\neq 0} c_{i,k} e^{i\lambda_{q+1}k\cdot \Phi_i}\, ,
\end{align*}
where
\[
c_{i,k}(x,t):=\frac{-i}{\lambda_{q+1}}\nabla( \left(\rho_{q,i}(x,t)\right)^{\sfrac12} a_{k}(\tilde R_{q,i}(x,t)))\times \frac{\nabla\Phi_i(x,t)^T(k\times A_k)}{\abs{k}^2}.
\]

Then from \eqref{e:Mikado_Potential} and the identity (see for instance \cite{DaSz2016})
$$
\curl\left(\nabla\Phi^TU(\Phi_i)\right)=\textrm{cof}\, \nabla\Phi^T(\curl U)(\Phi)=\nabla\Phi^{-1}(\curl U)(\Phi)
$$
one can check that
\begin{align}\label{e:w_curl}
w_{q+1} = w_o+w_c=\frac{-1}{\lambda_{q+1}}\curl\left(\sum_{i,k\neq 0} (\nabla\Phi_i)^T\left(\frac{ik\times b_{k,i}}{\abs{k}^2}\right)e^{i\lambda_{q+1}k\cdot \Phi_i}\right)\,.
\end{align}

Upon letting
\begin{align*}
\overline R_q = \sum_{i} R_{q,i}
\end{align*}

\subsection{The final Reynolds stress}
The new Reynolds stress is thus defined as 
\begin{align}
\mathring{R}_{q+1} =  \underbrace{\RR \left( w_{q+1} \cdot \nabla \overline v_q\right)}_{\mbox{Nash error}} + \underbrace{\RR \left( \partial_t  w_{q+1} + \overline v_q \cdot \nabla w_{q+1} \right)}_{\mbox{Transport error}} + \underbrace{\RR \div \left(- {\overline R}_{q} + (w_{q+1} \otimes w_{q+1}) \right)}_{\mbox{Oscillation error}}.
 \label{e:R:q+1} 
\end{align}
Notice that all three terms in \eqref{e:R:q+1} are of the form $\RR f$, where $f$ is either a divergence or a curl, and thus has zero mean.
With this definition and Proposition~\ref{p:R}, one may verify that 
\begin{equation*}
\left\{
\begin{array}{l}
 \partial_t v_{q+1} + \div (v_{q+1} \otimes v_{q+1}) + \nabla p_{q+1} = \div(\mathring{R}_{q+1}) \, ,
\\ \\
 \div v_{q+1} = 0 \, ,
\end{array}\right.
\end{equation*}
where the new pressure is defined by
\begin{equation}\label{e:new_pressure}
p_{q+1} = \bar p_q + |w_{q+1}|^2 - \sum_{i} \rho_{q,i} \, .
\end{equation}

\subsection{Estimates on the perturbation}

\begin{proposition}
\label{p:perturbation}
For $t\in \tilde I_i$ and any $N\geq 0$
\begin{align}
\norm{ (\nabla\Phi_i)^{-1}}_N + \norm{\nabla\Phi_i}_N &\lesssim \ell^{-N} \,,\label{e:phi_N}\\
\norm{\tilde R_{q,i}}_N &\lesssim  \ell^{-N}\,,\label{e:tR_est}\\
\norm{b_{i,k}}_N &\lesssim \delta_{q+1}^{\sfrac12}|k|^{-6}\ell^{-N}\,, \label{e:b_k_est_N}\\
\norm{c_{i,k}}_N &\lesssim  \delta_{q+1}^{\sfrac12}\lambda_{q+1}^{-1}|k|^{-6}\ell^{-N-1}\,.\label{e:c_k_est}
\end{align}
\end{proposition}

It is important to notice that the symbol $\lesssim$ denotes a dependence of the constants in the estimates from $N$, $\alpha$, $\beta$ and $M$, but not upon $k$ or $a$.

\begin{proof}[Proof of Proposition~\ref{p:perturbation}]
From \eqref{e:overline_v_N}, \eqref{e:Dphi_near_id} and \eqref{e:Dphi_N} we obtain
\begin{equation*}
\norm{\nabla\Phi_i}_N \lesssim 1+\tau_q\norm{D\overline v_q}_N \lesssim  1+ \tau_q\delta_q^{\sfrac 12}\lambda_q \ell^{-N}\, .
\end{equation*}
Using the fact that $\|\nabla\Phi_i - {\rm Id}\|_0 \leq \sfrac{1}{2}$ (see \eqref{e:Phi-close-to-id}), the estimate 
\eqref{e:phi_N}  follows (indeed it gives the slightly better estimate $\lesssim 1 + \ell^{-N+2\alpha}$, but the other is still enough for our purposes).

Recalling property (iv) of $\eta_i$ we see that  $\rho_{q,i}$ is a function of $t$ only on $\supp \mathring{\overline R_q}$, i.e.
$$
\rho_{q,i}(x,t)=\frac{\rho_q(t)}{\sum_j\int_{\T^3}\eta_j^2(y,t)\,dy}.
$$
Thus,
\begin{align}
 \frac{R_{q,i}}{\rho_{q,i}} = \Id -\frac{\sum_j\int_{\T^3}\eta_j^2(y,t)\,dy}{\rho_q(t)} \mathring{\overline R_q},\label{e:R/rho}
\end{align}
so that by \eqref{e:rho_range} and \eqref{e:Rq:1} we obtain
\begin{align}\label{e:ratio_later_needed}
 \norm{\frac{R_{q,i}}{\rho_{q,i}}}_N \lesssim 1+ \frac{\lambda_q^\alpha}{\delta_{q+1}} \norm{ \mathring{\overline R_q}}_N\ell^{-N} \lesssim \ell^{-N},
\end{align}
where we have applied the crude estimate $\lesssim 1+ \|\mathring{\overline R_q}\|_{N+\alpha} \lambda_q^\alpha \delta_{q+1}^{-1} \lesssim 1+ \ell^{-N +\alpha} \lambda_q^\alpha \lesssim \ell^{-N}$.

Therefore, using Lemma \ref{l:R_in_range} and property (v):
\begin{align*}
\norm{\tilde R_{q,i}}_N \lesssim &\norm{\nabla\Phi_i}_N\norm{\nabla\Phi_i}_0+ \norm{\frac{R_{q,i}}{ \rho_{q,i}}}_N 
\lesssim  \norm{\nabla\Phi_i}_N\norm{\nabla\Phi_i}_0+ \ell^{-N}\, .
\end{align*}
The estimate \eqref{e:tR_est} then follows from \eqref{e:phi_N}. 

The estimate \eqref{e:b_k_est_N} follows as a consequence of \eqref{e:a_k_est}, \eqref{e:rho_i_bnd} and \eqref{e:tR_est}. The estimate \eqref{e:c_k_est} follows as a consequence of \eqref{e:a_k_est}, \eqref{e:rho_i_bnd}, \eqref{e:phi_N} and \eqref{e:tR_est}.
\end{proof}

\begin{corollary}
\label{c:perturbation}
Assuming $a$ is sufficiently large, the perturbations $w_o$, $w_c$ and $w_q$ satisfy the following estimates
\begin{align}
\norm{w_o}_0 +\frac{1}{\lambda_{q+1}}\norm{w_o}_1 &\leq \frac{M}{4}\delta_{q+1}^{\sfrac 12}\label{e:w_o_est}\\
\norm{w_c}_0+\frac{1}{\lambda_{q+1}} \norm{w_c}_1 &\lesssim \delta_{q+1}^{\sfrac 12}\ell^{-1}\lambda_{q+1}^{N-1}\label{e:w_c_est}\\
\norm{w_{q+1}}_0 +\frac{1}{\lambda_{q+1}}\norm{w_{q+1}}_1 &\leq  \frac{M}{2} \delta_{q+1}^{\sfrac 12}\label{e:w_est}
\end{align}
where the constant $M$ depends solely on the constant $c_0$ in \eqref{e:sum_eta_range}.
In particular, we obtain \eqref{e:outline_v_diff}.
\end{corollary}
\begin{proof}[Proof of Corollary~\ref{c:perturbation}]
Taking into account \eqref{e:Phi-close-to-id}, we conclude $\|(\nabla \Phi_i)^{-1}\|_0 \leq 2$ on $\supp (\eta_i)$.
Thus, taking into account that the $w_{o,i}$ have disjoint supports, from Lemma \ref{l:choice_of_M} we conclude
\begin{equation}\label{e:M_0}
\|w_o\| \leq 2 \delta_{q+1}^{\sfrac{1}{2}} \sum_{k\neq 0} \frac{\bar M}{|k|^4} \leq \frac{M}{32}\, .
\end{equation}
To estimate $\norm{w_o}_1$ we observe first that
\begin{equation}\label{e:exp_est}
\norm{\nabla (e^{i\lambda_{q+1}k\cdot \Phi_i})}_0\leq \lambda_{q+1}\abs{k}\norm{\nabla\Phi_i}_{0}\leq 2\lambda_{q+1}\abs{k}\,.
\end{equation}
Compute now 
\[
\nabla w_{o,i} = \sum_k (\nabla \Phi_i)^{-1} b_{i,k} \nabla (e^{i \lambda_{q+1} k\cdot \Phi_i}) +
\sum_k \nabla ((\nabla \Phi_i)^{-1} b_{i,k}) e^{i \lambda_{q+1} k\cdot \Phi_i}\, .
\]
In particular, from \eqref{e:exp_est}, Lemma \ref{l:choice_of_M} and Proposition \ref{p:perturbation} (taking into account that the supports of the $w_{o,i}$ are disjoint), we conclude 
\begin{align}
\|\nabla w_{o}\|_0 \leq & 4 \delta_{q+1}^{\sfrac{1}{2}} \lambda_{q+1}  \sum_{k\neq 0} \frac{\bar M}{|k|^4} + C \delta_{q+1}^{\sfrac{1}{2}} \ell^{-1} \sum_{k\neq 0} \frac{1}{|k|^6} \leq \frac{M}{16} \delta_{q+1}^{\sfrac{1}{2}} \lambda_{q+1} + \bar C \delta_{q+1}^{\sfrac{1}{2}} \ell^{-1}\, ,\label{e:M_1}
\end{align}
where the constant $\bar C$ depends upon $\beta, \alpha$ and $M$, but not upon $a$. In particular, summing \eqref{e:M_0} and \eqref{e:M_1} we achieve
\begin{equation}\label{e:still_M}
\|w_o\|_0 + \lambda_{q+1}^{-1} \|w_o\|_1 \leq \frac{M}{8} \delta_{q+1}^{\sfrac{1}{2}} + \bar C (\lambda_{q+1} \ell)^{-1}\, .
\end{equation}

By our definition of the various parameters we get
\[
(\ell \lambda_{q+1})^{-1} = \frac{\delta_q^{\sfrac{1}{2}} \lambda_q^{1+\sfrac{3\alpha}{2}}}{\delta_{q+1}^{\sfrac{1}{2}} \lambda_{q+1}} = \frac{\lambda_q^{1-\beta + \sfrac{3\alpha}{2}}}{\lambda_{q+1}^{1-\beta}}
\leq \tilde C \delta_{q+1}^{\sfrac{1}{2}} a^{b^q (1-\beta + \sfrac{3\alpha}{2}) - b^{q+1} (1-\beta)}\, 
\]
where the constant $\tilde C$ depends on \eqref{e:bloody_integers}. Having chosen $\alpha$ small enough so that $b> \sfrac{1-\beta + 3\alpha/2}{1-\beta}$, for $a$ sufficiently large we achieve that the right hand side of \eqref{e:still_M} is smaller than $\sfrac{M}{4} \delta_{q+1}^{\sfrac{1}{2}}$.

The estimate \eqref{e:w_c_est} follows as a direct consequence of \eqref{e:c_k_est}  and \eqref{e:exp_est}.

Combining \eqref{e:w_o_est} and \eqref{e:w_c_est} we achieve 
\begin{equation}\label{e:not_fishy}
\norm{w_{q+1}}_0 +\frac{1}{\lambda_{q+1}}\norm{w_{q+1}}_1 \leq \delta_{q+1}^{\sfrac{1}{2}}
\left(\frac{M}{4} + C (\ell \lambda_{q+1})^{-1}\right)\, ,
\end{equation}
where the constant $C$  depends upon $\beta, \alpha$ and $M$, but not upon $a$. Hence, arguing as above,  if $b> \sfrac{1-\beta + 3\alpha/2}{1-\beta}$ then \eqref{e:w_est} holds for $a$ sufficiently large (depending on $\beta, \alpha$ and $M$).
\end{proof}
Let us define $D_{t,q} : =\partial_t +\overline v_q\cdot \nabla$ to be the material derivative associated with $\overline v_q$. We then have
\begin{proposition}
For $t\in \tilde I_i$ and $N\geq 0$ we have
\label{p:D_t_estimates}
\begin{align}
\norm{D_{t,q} \nabla\Phi_i}_N &\lesssim  \delta_q^{\sfrac 12}\lambda_q  \ell^{-N}  \label{e:D_t:phi_N}\\
\norm{D_{t,q} \tilde R_{q,i}}_N &\lesssim  \tau_q^{-1}\ell^{-N}\label{e:D_t:tR}\\
\norm{D_{t,q} c_{i,k}}_N &\lesssim  \delta_{q+1} ^{\sfrac 12}\tau_q^{-1}\ell^{-N-1}\lambda_{q+1}^{-1}|k|^{-6} \, .\label{e:D_t:c_k}
\end{align}

\end{proposition}

\begin{proof}[Proof of Proposition~\ref{p:D_t_estimates}]
Observe that
\[D_{t,q} \nabla\Phi_i =-\nabla\Phi_i D\overline v_q\]
In particular, 
\[
\|D_{t,q} \nabla\Phi_i\|_N \lesssim \|\nabla\Phi_i\|_0 \|\overline v_q\|_{N+1} + \|\nabla\Phi_i\|_{N} \|\overline v_q\|_1 \,.
\]
Thus \eqref{e:D_t:phi_N} follows from  \eqref{e:vq:1} and \eqref{e:phi_N}. 
Next, we observe that 
\[
D_{t,q} \rho_{q,i} = \partial_t \rho_{q,i} + \bar{v}_q \cdot \nabla \rho_{q,i}
\]
and thus we can estimate
\[
\|D_{t,q} \rho_{q,i}\|_N \lesssim \|\partial_t \rho_{q,i}\|_N + \|\rho_{q,i}\|_{N+1} \|\bar v_q\|_0 + \|\bar v_q\|_N \|\rho_{q,i}\|_1.
\]
Recall that $\|\bar v_q\|_0 \leq \|v_\ell\|_0 + \|v_\ell - v_q\|_0 \lesssim 1 \lesssim \tau_q^{-1}$ and so from \eqref{e:vq:1} we conclude $\|\bar v_q\|_N \leq \tau_q^{-1} \ell^{-N}$. Combining the latter estimate with \eqref{e:rho_i_bnd_N} and \eqref{e:rho_i_bnd_t} we achieve 
\begin{equation}
\|D_{t,q} \rho_{q,i}\|_N \lesssim \delta_{q+1} \tau_q^{-1} \ell^{-N}\, .
\end{equation}
Differentiating \eqref{e:R/rho} we have
\begin{equation}
D_{t,q} (\rho_{q,i}^{-1} R_{q,i}) = - \left(\partial_t \frac{\sum_j\int_{\T^3}\eta_j^2(y,t)\,dy}{\rho_q(t)}\right) \mathring{\overline R_q} - \frac{\sum_j\int_{\T^3}\eta_j^2(y,t)\,dy}{\rho_q(t)} D_{t,q} \mathring{\overline R}_q \, .
\end{equation}
Thus we can estimate, using \eqref{e:Rq:1} and \eqref{e:Rq:Dt}: 
\begin{align}
\|D_{t,q} (\rho_{q,i}^{-1}R_{q,i})\|_N 
&\lesssim \delta_{q+1}^{-1} \delta_q^{\sfrac{1}{2}} \lambda_q^{1+2\alpha} \|\mathring{\overline R_q}\|_N + \tau_q^{-1} \delta_{q+1}^{-1} \lambda_q^\alpha \|\mathring{\overline R_q}\|_N + \delta_{q+1}^{-1} \lambda_q^\alpha \|D_{t,q} \mathring{\overline R_q}\|_N\notag\\
&\lesssim  \delta_q^{\sfrac{1}{2}} \lambda_q^{1+2\alpha} \ell^{-N+\alpha} + \tau_q^{-1} \lambda_q^\alpha \ell^{-N+\alpha} + \lambda_q^\alpha \delta_q^{\sfrac{1}{2}} \lambda_q \ell^{-N-\alpha}\lesssim
\tau_q^{-1} \ell^{-N}\, .\label{e:Dt_ratio}
\end{align}

Differentiating \eqref{e:tildeR_def} we achieve
\[
D_{t,q}  \tilde R_{q,i} = D_{t,q} \nabla\Phi_i (\rho_{q,i}^{-1} R_{q,i}) \nabla\Phi_i^T +
\nabla\Phi_i D_{t,q} (\rho_{q,i}^{-1} R_{q,i}) \nabla\Phi_i^T + \nabla\Phi_i (\rho_{q,i}^{-1} R_{q,i})(D_{t,q} \nabla\Phi_i)^T\, .
\]
Thus we can estimate
\begin{align*}
\|D_{t,q} \tilde R_{q,i}\|_N \lesssim & \|D_{t,q} \nabla\Phi_i\|_N \|(\rho_{q,i}^{-1} R_{q,i})\|_0 +
\|D_{t,q} \nabla\Phi_i\|_0 \|(\rho_{q,i}^{-1} R_{q,i})\|_N\\
& + \|D_{t,q} \nabla\Phi_i \|_0 \|(\rho_{q,i}^{-1} R_{q,i})\|_0 \|\nabla\Phi_i\|_N
+ \|D_{t,q} (\rho_{q,i}^{-1} R_{q,i})\|_N + \|D_{t,q} (\rho_{q,i}^{-1} R_{q,i})\|_0 \|\nabla\Phi_i\|_N\, .
\end{align*}
Using \eqref{e:D_t:phi_N}, \eqref{e:Dt_ratio}, \eqref{e:ratio_later_needed} and \eqref{e:phi_N}, we  conclude \eqref{e:D_t:tR}.

Finally, the estimate \eqref{e:D_t:c_k} follows as a consequence of \eqref{e:a_k_est}, Lemma \ref{l:R_in_range}, Proposition \ref{p:perturbation}, \eqref{e:D_t:phi_N}, and \eqref{e:D_t:tR}.
\end{proof}


\section{Proof of Proposition \ref{p:main}}
\label{sec:p_main}

In this section we complete the proof of Proposition \ref{p:main} by proving the remaining estimates \eqref{e:outline_R_est} and \eqref{e:outline_energy_diff}.

\subsection{Estimates of the new Reynolds stress error}

In the proposition below we prove the inductive estimates on $\mathring R_{q+1}$:
\begin{proposition}
\label{p:R_q+1}
The Reynolds stress error $\mathring R_{q+1}$ defined in \eqref{e:R:q+1} satisfies the estimate
\begin{equation}\label{e:final_R_est}
\norm{\mathring R_{q+1}}_{0}\lesssim \frac{\delta_{q+1}^{\sfrac12}\delta_q^{\sfrac{1}{2}} \lambda_q}{\lambda_{q+1}^{1-4\alpha}} \,.
\end{equation}
In particular, \eqref{e:outline_R_est} holds.
\end{proposition}

\subsubsection{Nash error}

We just write this term as
\begin{align*}
\mathcal R\left(w_{q+1} \cdot \nabla \overline v_q \right)
=\sum_{i}\sum_{k\neq0}{\mathcal R}\left( \Big( (\nabla\Phi_i)^{-1} b_{i,k} e^{i\lambda_{q+1}k\cdot \Phi_i} +c_{i,k} e^{i\lambda_{q+1}k\cdot \Phi_i} \Big)\cdot \nabla \overline v_{q}  \right).
\end{align*}
Using Proposition~\ref{p:R:oscillatory_phase} we bound for $t\in \tilde I_i$
\begin{align*}
 &\norm{{\mathcal R}\left(  (\nabla\Phi_i)^{-1} b_{i,k} e^{i\lambda_{q+1}k\cdot \Phi_i}\cdot \nabla \overline v_{q} \right)}_{\alpha}\notag\\ 
 &\qquad \lesssim \frac{\norm{\nabla\Phi_i^{-1} b_{i,k}\cdot \nabla \overline v_{q}}_{0}}{\lambda_{q+1}^{1-\alpha}}
+ \frac{\norm{ \nabla\Phi_i^{-1} b_{i,k}\cdot \nabla \overline v_{q}}_{N+\alpha}+\norm{\nabla\Phi_i^{-1} b_{i,k}\cdot \nabla \overline v_{q}}_{0}\norm{ \Phi_i}_{N+\alpha}}{\lambda_{q+1}^{N-\alpha}}
 \notag\\
 &\qquad \lesssim \frac{\lambda_q \delta_{q+1}^{\sfrac12} \delta_q^{\sfrac12}}{\lambda_{q+1}^{1-\alpha}|k|^6} + \frac{ \lambda_{q}\delta_{q+1}^{\sfrac12} \delta_q^{\sfrac12}}{\lambda_{q+1}^{N-\alpha}\ell^{N+\alpha} |k|^6}\, .
\end{align*}
Now, provided $\alpha$ is sufficiently small we claim that we can first fix a suitable $N$ and then choose $a$ large enough, so that
\[
\frac{1}{\lambda_{q+1}^{N-\alpha}\ell^{N+\alpha}} \leq \frac{1}{\lambda_{q+1}^{1-\alpha}}\, .
\]
Such choice is equivalent to $\lambda_{q+1}^{(N-1)-(N-\alpha) \beta} \geq \lambda_q^{(1-\beta +3\alpha/2)(N+\alpha)}$. 
Taking the logarithms in base $a$, we need the condition 
\[
b^{q+1} ((N-1) - (N-\alpha) \beta) > b^q (N+\alpha) \left(1 - \beta +\frac{3\alpha}{2}\right)\, , 
\] 
which would determine the needed $N$. In order to show that for $\alpha$ sufficiently small we can choose such an $N$, we just need to verify the existence of $N$ such that $b ((N-1) -N \beta)> N (1-\beta)$. The latter is equivalent to $(b-1)(N-1) (1-\beta) > (1-\beta) + b \beta$ which in turn, since $\beta < \sfrac{1}{3}$ and $b>1$, can certainly be satisfied for $N$ large enough. 
Finally, having chosen first $\alpha>0$ and then $N$ according to the above requirement, we can then take $a\gg 1$ large enough to beat the eventual geometric constant due to \eqref{e:bloody_integers}.
Hence we achieve
\begin{equation}\label{e:Nash_freq}
\norm{{\mathcal R}\left(  (\nabla\Phi_i)^{-1} b_{i,k} e^{i\lambda_{q+1}k\cdot \Phi_i}\cdot \nabla \overline v_{q} \right)}_{\alpha} \lesssim \frac{\lambda_q \delta_{q+1}^{\sfrac12} \delta_q^{\sfrac12}}{\lambda_{q+1}^{1-\alpha}|k|^6}\, .\end{equation}

For the second term in the Nash error we 
again use Corollary~\ref{p:R:oscillatory_phase} to obtain
\begin{align}
\norm{{\mathcal R}\left( c_{i,k} e^{i\lambda_{q+1}k\cdot \Phi_i}\cdot \nabla \overline v_{q} \right)}_{\alpha} 
 &\lesssim \frac{\norm{c_{i,k}\cdot \nabla \overline v_{q}}_{0}}{\lambda_{q+1}^{1-\alpha}} + \frac{\norm{ c_{i,k}\cdot \nabla \overline v_{q}}_{N+\alpha}+\norm{ c_{i,k}\cdot \nabla \overline v_{q}}_{0}\norm{ \Phi_i}_{N+\alpha}}{\lambda_{q+1}^{N-\alpha}}
 \notag\\
 &\lesssim \frac{\delta_{q+1}^{\sfrac12}\delta_{q}^{\sfrac12}\lambda_q}{\ell\lambda_{q+1}^{2-\alpha}|k|^{6}}+\frac{\delta_{q+1}^{\sfrac12}\delta_{q}^{\sfrac12}\lambda_q}{\ell^{N+1-\alpha}\lambda_{q+1}^{N+1-\alpha}|k|^{6}}
 \lesssim \frac{\delta_{q+1}^{\sfrac12}\delta_{q}^{\sfrac12}\lambda_q}{\lambda_{q+1}^{1-\alpha}|k|^{6}}\label{e:Nash_freq_2}
\end{align}
where again we assume to have fixed first $N$ and then $a$ large enough. 
We also implicitly used that 
\begin{equation}\label{e:ell-lambdaq+1}
\ell \lambda_{q+1} \geq 1\,,
\end{equation}
which is equivalent to 
$\lambda_{q+1}^{1-\beta} \geq \lambda_q^{1-\beta + \sfrac{3\alpha}{2}}$. The latter inequality follows from~\eqref{e:b_beta_rel} and \eqref{e:bloody_integers}, upon taking logarithms in base $a$, choosing first $\alpha$ so that $(b - 1)(1-\beta) \geq 3\alpha $, and then $a$ sufficiently large so that $\tfrac{(b-1)}{10} \geq \log_a (4\pi)$.

Summing over the frequencies and using that $\sum_{k\in \Z^3\setminus \{0\}} |k|^{-6} < \infty$, we achieve
\begin{equation}\label{e:Nash_est}
\mathcal R\left(w_{q+1} \cdot \nabla \overline v_q \right) \lesssim \frac{ \delta_{q+1}^{\sfrac12} \delta_q^{\sfrac12} \lambda_q}{\lambda_{q+1}^{1-\alpha}|k|^6}\, .
\end{equation}

\subsubsection{Transport error}
We split the transport error into two parts
\[(\partial_t+\overline v_q\cdot \nabla) w_{q+1}=(\partial_t+\overline v_q\cdot \nabla) w_o+(\partial_t+\overline v_q\cdot \nabla) w_c.\]

Applying \eqref{e:Lie_advection} yields
\begin{equation}\label{e:wotransport}
\begin{split}
(\partial_t+\overline v_q\cdot \nabla) w_o =&
\sum_{i,k} (\nabla\overline v_q)^T(\nabla\Phi_i)^{-1} b_{i,k} e^{i\lambda_{q+1}k\cdot \Phi_i}\\&\quad +
\sum_{i,k} (\partial_t+\overline v_q\cdot \nabla) \left(\rho_{q,i}^{\sfrac12} a_k(\tilde R_i)\right) (\nabla\Phi_i)^{-1} A_k e^{i\lambda_{q+1}k\cdot \Phi_i} \,.
\end{split}
\end{equation}
We then apply Corollary~\ref{p:R:oscillatory_phase} to obtain for the first term in \eqref{e:wotransport}
\begin{align}
&\norm{\mathcal R\left( (\nabla\overline v_q)^T(\nabla\Phi_i)^{-1} b_{i,k} e^{i\lambda_{q+1}k\cdot \Phi_i} \right)}_{\alpha}\notag\\
&\lesssim 
 \frac{\norm{ (\nabla \overline v_q)^T(\nabla\Phi_i)^{-1} b_{i,k}}_{0}}{\lambda_{q+1}^{1-\alpha}}+\frac{\norm{  (\nabla \overline v_q)^T(\nabla\Phi_i)^{-1} b_{i,k}}_{N+\alpha}}{\lambda_{q+1}^{N-\alpha}}
+\frac{\norm{  (\nabla \overline v_q)^T(\nabla\Phi_i)^{-1} b_{i,k}}_{0}\norm{  \Phi_i}_{N+\alpha}}{\lambda_{q+1}^{N-\alpha}} \,.\label{e:transport_1}
\end{align}
We use Proposition \ref{p:perturbation} and Proposition \ref{p:vq:vell} to estimate
\begin{align*}
\norm{  (\nabla \overline v_q)^T(\nabla\Phi_i)^{-1} b_{i,k}}_{N+\alpha}
&\lesssim \|\nabla \overline v_q\|_{N+\alpha} \|(\nabla\Phi_i)^{-1}\|_\alpha
\|b_{i,k}\|_\alpha\notag\\
& +  \|\nabla \overline v_q\|_{\alpha} \|(\nabla\Phi_i)^{-1}\|_{N+\alpha}
\|b_{i,k}\|_\alpha +  \|\nabla \overline v_q\|_{\alpha} \|(\nabla\Phi_i)^{-1}\|_\alpha
\|b_{i,k}\|_{N+\alpha}\\
&\lesssim \delta_{q+1}^{\sfrac{1}{2}} \ell^{-N-3\alpha} \lesssim \delta_{q+1}^{\sfrac{1}{2}} \delta_q^{\sfrac{1}{2}} \lambda_q \ell^{-N-1-3\alpha}\, .
\end{align*}
Arguing in a similar fashion for the third summand in \eqref{e:transport_1}, we achieve
\begin{align*}
\norm{\mathcal R\left( (\nabla\overline v_q)^T(\nabla\Phi_i)^{-1} b_{i,k} e^{i\lambda_{q+1}k\cdot \Phi_i} \right)}_{\alpha} &\lesssim \frac{\lambda_q \delta_{q+1}^{\sfrac12} \delta_q^{\sfrac12}}{\lambda_{q+1}^{1-\alpha}|k|^6} + \frac{ \lambda_{q}\delta_{q+1}^{\sfrac12} \delta_q^{\sfrac12}}{\lambda_{q+1}^{N-\alpha} \ell^{N+1+3\alpha} |k|^6}
 \lesssim \frac{\lambda_q \delta_{q+1}^{\sfrac12} \delta_q^{\sfrac12}}{\lambda_{q+1}^{1-\alpha}|k|^6}\, ,
\end{align*}
where in the last inequality, as in the previous section, we have assumed $\alpha$ sufficiently small and $N$ appropriately chosen. 

For the second term in \eqref{e:wotransport}, let us define
\[  
d_{i,k}(x,t):=D_{t,q}\left(\left(\rho_{q+1,i}(x,t)\right)^{\sfrac12} a_k(\tilde R_i(x,t))\right) (\nabla\Phi_i(x,t))^{-1} A_k\,.
\]
Using \eqref{e:phi_N}, \eqref{e:rho_i_bnd_N}, \eqref{e:rho_i_bnd_t} and \eqref{e:D_t:tR} and again assuming $N$ sufficiently large, arguing as above we conclude
\begin{align*}
\norm{\mathcal R\left( d_{i,k}e^{i\lambda_{q+1}k\cdot \Phi_i} \right)}_{\alpha}
&\lesssim
\frac{\norm{ d_{i,k }}_{0}}{\lambda_{q+1}^{1-\alpha}} +
\frac{\norm{ d_{i,k }}_{N+\alpha}+\norm{ d_{i,k }}_{0}\norm{\Phi_i}_{N+\alpha}}{\lambda_{q+1}^{N-\alpha}}\\
&\lesssim \frac{\delta_{q+1} ^{\sfrac 12}}{\tau_q\lambda_{q+1}^{1-\alpha}|k|^{6}}
= \frac{\delta_{q+1}^{\sfrac12}\delta_{q}^{\sfrac12}\lambda_q}{\lambda_{q+1}^{1-\alpha}}\ell^{-2\alpha}|k|^{-6}
\lesssim \frac{\delta_{q+1}^{\sfrac12}\delta_{q}^{\sfrac12}\lambda_q}{\lambda_{q+1}^{1-4\alpha}} |k|^{-6}\, ,
\end{align*}
where we have used $\ell^{-2\alpha} \leq \lambda_q^{3\alpha}\leq \lambda_{q+1}^{3\alpha}$ (see \eqref{e:compare-lambda-ell}).

Now we consider the term involving the material derivative of the correction. Observe
\begin{align*}
(\partial_t+\overline v_q\cdot \nabla) w_c =&
\sum_{i,k} \left(D_{t,q} c_{i,k }\right)e^{i\lambda_{q+1}k\cdot \Phi_i}
\end{align*}
Then applying Corollary~\ref{p:R:oscillatory_phase} and \eqref{e:D_t:c_k} yields
\begin{align*}
\norm{\mathcal R \left(\left(D_{t,q} c_{i,k }\right)e^{i\lambda_{q+1}k\cdot \Phi_i}\right)}_{\alpha}\lesssim &
\frac{\norm{D_{t,q} c_{i,k }}_{0}}{\lambda_{q+1}^{1-\alpha}} +
\frac{\norm{D_{t,q} c_{i,k }}_{N+\alpha}+\norm{D_{t,q} c_{i,k }}_{0}\norm{\Phi_i}_{N+\alpha}}{\lambda_{q+1}^{N-\alpha}}\\
\lesssim & \frac{\delta_{q+1} ^{\sfrac 12}}{\tau_q\ell \lambda_{q+1}^{2-\alpha}|k|^{6}}
\lesssim  \frac{\delta_{q+1} ^{\sfrac 12}}{\tau_q\lambda_{q+1}^{1-\alpha}|k|^{6}}
\lesssim \frac{\delta_{q+1}^{\sfrac12}\delta_{q}^{\sfrac12}\lambda_q}{\lambda_{q+1}^{1-3\alpha}}|k|^{-6}
\end{align*}
where we used \eqref{e:phi_N}, \eqref{e:D_t:c_k} and \eqref{e:ell-lambdaq+1}.

Again, summing over $k \neq 0$ we reach the inequality
\begin{equation}\label{e:trans_est}
\|\RR \left( \partial_t  w_{q+1} + \overline v_q \cdot \nabla w_{q+1} \right)\|_\alpha \lesssim
\frac{\delta_{q+1}^{\sfrac12}\delta_{q}^{\sfrac12}\lambda_q}{\lambda_{q+1}^{1-3\alpha}}\, .
\end{equation}

\subsubsection{Oscillation error}

Recall the oscillation error may be written as
\begin{align*}
&\RR \div \left(- {\overline R}_q + w_{q+1}  \otimes  w_{q+1} \right) \notag\\
&\quad = \underbrace{\RR \div\left (- {\overline R}_q  + w_o\otimes w_o\right)}_{=: \mathcal O_1}+
\underbrace{\mathcal R \div \left (w_o\otimes w_c+w_c\otimes w_o+w_c \otimes w_c\right)}_{=:\mathcal O_2}\,.
\end{align*}
For the second term we proceed as follows:
\begin{align}
\norm{\mathcal O_2}_{\alpha} \lesssim& \norm{w_o\otimes w_c+w_c\otimes w_o+w_c\otimes w_c}_\alpha\notag\\
\lesssim & \norm{w_0}_0\norm{w_c}_{\alpha}+\norm{w_0}_{\alpha}\norm{w_c}_{0}+\norm{w_c}_{\alpha}^2
\lesssim \frac{\delta_{q+1}}{\ell\lambda_{q+1}^{1-\alpha}} 
\lesssim  \frac{\delta_{q+1}^{\sfrac12}\delta_{q}^{\sfrac12}\lambda_q}{\lambda_{q+1}^{1-\alpha}} \label{e:O1}
\,.
\end{align}

Now consider $\mathcal O_1$. Due to the supports of the cutoffs $\eta_j$ being mutually disjoint, we have
\begin{align*}
 \mathcal O_1 = \RR \div \left(- {\overline R}_q +  \sum_{i} w_{o,i}  \otimes w_{o,i}  \right).
\end{align*}
Using the definition of $w_{o,i}$ and \eqref{e:Mikado_stationarity}, the tensor $w_{o,i} \otimes w_{o,i} $ may be written as
\begin{align}
 w_{o,i} \otimes w_{o,i} &= \rho_{q,i} \nabla \Phi_i^{-1} (W\otimes W)(\tilde R_{q,i}, \lambda_{q+1} \Phi_i) \nabla \Phi_i^{-T} 
 \notag \\
 &= \rho_{q,i} \nabla \Phi_i^{-1} \tilde R_{q,i} \nabla \Phi_i^{-T} + \sum_{k\neq 0} \rho_{q,i}  \nabla \Phi_i^{-1}  C_{k}(\tilde R_{q,i}) \nabla \Phi_i^{-T} e^{i\lambda_{q+1}k\cdot \Phi_i} \notag \\
 &= R_{q,i}    + \sum_{k\neq 0}\rho_{q,i}  \nabla \Phi_i^{-1}  C_{k}(\tilde R_{q,i}) \nabla \Phi_i^{-T} e^{i\lambda_{q+1}k\cdot \Phi_i}. \label{e:wowo}
\end{align}

On the other hand, recalling \eqref{e:Ck_ind} 
\[
\nabla\Phi_i^{-1}C_k \nabla\Phi_i^{-T}\nabla\Phi_i^{T}k=0,
\] 
consequently
\begin{align*}
 \div\left(\sum_{i} w_{o,i} \otimes w_{o,i}-R_{q,i}  \right)=\sum_{i,k\neq 0} \div( \rho_{q,i}  \nabla\Phi_i^{-1} C_{k}(\tilde R_{q,i}) \nabla\Phi_i^{-T}) e^{i\lambda_{q+1}k\cdot \Phi_i} 
\end{align*}
Thus, by Proposition~\ref{p:R:oscillatory_phase} 
\begin{align}
\norm{\mathcal O_1}_{\alpha}
&\lesssim \sum_{i}\sum_{k\neq0} \frac{\norm{\div(\rho_{q,i}  \nabla\Phi_i^{-1} C_{k}(\tilde R_{q,i}) \nabla\Phi_i^{-T})}_0}{\lambda_{q+1}^{1-\alpha}} \notag\\
&\  + \sum_{i}\sum_{k\neq0}
\frac{\norm{\div( \rho_{q,i}  \nabla\Phi_i^{-1} C_{k}(\tilde R_{q,i}) \nabla\Phi_i^{-T})}_{N+\alpha}+\norm{\div( \rho_{q,i}  \nabla\Phi_i^{-1} C_{k}(\tilde R_{q,i}) \nabla\Phi_i^{-T})}_{0}\norm{\Phi_i}_{N+\alpha}}{\lambda_{q+1}^{N-\alpha}}\notag\\
&\lesssim \sum_{i}\sum_{k\neq 0}\frac{\delta_{q+1}}{\ell\lambda_{q+1}^{1-\alpha}|k|^{6}}
\lesssim \frac{\delta_{q+1}^{\sfrac12}\delta_q^{\sfrac12}\lambda_q}{\lambda_{q+1}^{1-\alpha}}
\,,\label{e:O2}
\end{align}
where we have used  \eqref{e:Ck_ind} and, as in the previous sections, a large choice of $N$ to absorb the estimates of the second line in that for the first line. Clearly, \eqref{e:O1} and \eqref{e:O2} give
\begin{equation}\label{e:osc_est}
\|\RR \div \left(- {\overline R}_q + w_{q+1}  \otimes  w_{q+1}\right)\|_\alpha \lesssim
\frac{\delta_{q+1}^{\sfrac12}\delta_q^{\sfrac12}\lambda_q}{\lambda_{q+1}^{1-\alpha}}\, .
\end{equation}

\subsubsection{Conclusion} Clearly \eqref{e:final_R_est} follows from \eqref{e:Nash_est}, \eqref{e:trans_est}, \eqref{e:osc_est} and \eqref{e:R:q+1}. 

\subsection{Energy iterate}

\eqref{e:energy_inductive_assumption} in the following proposition: 
\begin{proposition}\label{p:energy}
 The energy of $v_{q+1}$ satisfies the following estimate:
\begin{equation*}
    \abs{e(t)-\int_{\T^3}\abs{v_{q+1}}^2\,dx-\frac{\delta_{q+2}}2 }\lesssim \frac{\delta_q^{\sfrac12}\delta_{q+1}^{\sfrac12}\lambda_q^{1+2\alpha}}{\lambda_{q+1}}\,.
\end{equation*}
In particular, the estimate \eqref{e:outline_energy_diff} holds.
\end{proposition}
\begin{proof}[Proof of Proposition~\ref{p:energy}]
By definition we have
\begin{equation*}
\int_{\T^3}\abs{v_{q+1}}^2\,dx=\int_{\T^3}\abs{\overline v_{q}}^2\,dx + 2\int_{\T^3} w_{q+1}\cdot \overline v_q\,dx+\int_{\T^3}\abs{w_{q}(x,t)}^2\,dx
\end{equation*}

We also recall that
\begin{equation*}
\sum_i \int_{\T^3} \tr R_{q,i} (x,t)\,dx= 3 \sum_i \int_{\T^3} \rho_{q,i}(x,t) dx = 3 \rho_q(t) =e(t)-\frac{\delta_{q+2}}{2} - \int_{\T^3} \abs{\overline v_q}^2\,dx.
\end{equation*}

By integrating by parts once and using the identity \eqref{e:w_curl} and the estimates \eqref{e:phi_N} and \eqref{e:b_k_est_N} we obtain
\begin{align*}
\abs{\int_{\T^3} w_{q+1}\cdot \overline v_q\,dx}=&\sum_{i}\sum_{k\neq0}\norm{ (\nabla\Phi_i)^T\left(\frac{ik\times b_k}{\abs{k}^2}\right)}_0\norm{\overline v_q}_1
\lesssim \frac{\delta_q^{\sfrac12}\delta_{q+1}^{\sfrac12}\lambda_q}{\lambda_{q+1}} \,.
\end{align*}

Using \eqref{e:w_o_est} and \eqref{e:w_c_est} yields
\begin{align*}
\abs{\int_{\T^3} 2w_o\cdot w_c+\abs{w_c}^2\,dx}\lesssim \frac{\delta_{q+1}}{\ell\lambda_{q+1}}\stackrel{\eqref{e:ell_def}}= \frac{\delta_q^{\sfrac12}\delta_{q+1}^{\sfrac12}\lambda_q^{1+2\alpha}}{\lambda_{q+1}} \,.
\end{align*}

Finally, recall from \eqref{e:wowo} that 
\begin{align*}
\int_{\T^3} \abs{w_{o}(x,t)}^2\,dx &=
\sum_i \int_{\T^3} \tr R_{q,i} (x,t)\,dx+
\int_{\T^3}\sum_{i,k\neq 0} \rho_{q,i} \nabla\Phi_i^{-1}\tr C_k(\tilde R_{q,i}) \nabla\Phi_i^{-T}e^{i\lambda_{q+1} k \cdot \Phi_i}\,dx\\
\end{align*}
and thus it remains to bound the second term.
Set $e_{q,i}:=\rho_{q,i} \nabla\Phi_i^{-1}\tr C_k(\tilde R_i) \nabla\Phi_i^{-T}$ and use Proposition \ref{p:perturbation} and Lemma \ref{l:R_in_range} to conclude
\[
\|e_{q,i}\|_N \lesssim \delta_{q+1} \ell^{-N}\, .
\]
Next observe that at any given time at most two $e_{q,i}$ are nonvanishing. Hence use \eqref{e:phase_integration} in Proposition \ref{p:R:oscillatory_phase} to bound 

\begin{align*}
\left| \int_{\T^3}\sum_{i}\sum_{k\neq0}\rho_i \nabla\Phi_i^{-1}\tr C_k(\tilde R_i) \nabla\Phi_i^{-T}e^{i\lambda_{q+1} k \cdot \Phi_i}\,dx \right| \lesssim& 
\sum_{k\neq0}\frac{\delta_{q+1}\ell^{-N}}{\lambda_{q+1}^N|k|^{N}}.
\end{align*}
As already argued several time, we can choose $N$ such that $\sfrac{\delta_{q+1}\ell^{-N}}{\lambda_{q+1}^N}\leq \sfrac{\delta_{q+1} \delta_q^{\sfrac{1}{2}} \lambda_q}{\lambda_{q+1}}$. Assuming in addition that $N$ is larger than $4$ (so that the series is summable), we obtain the desired estimate.
\end{proof}

\section{An \texorpdfstring{$h$}{h}-principle}
\label{s:h}

In order to prove Theorem \ref{t:h-principle}, let us first state a variant of Proposition 3.1 from \cite{DaSz2016} that follows from the estimates in Section 5 used to prove the proposition in  \cite{DaSz2016}:

\begin{theorem}\label{t:first_step}
Let $(\bar{v},\bar{p},\bar{R})$ be a smooth strict subsolution of the Euler equations on $\T^3\times[0,T]$ and fix  $0<\gamma<1$. Then there exists $\eps_0>0$ such that for any $\eps<\eps_0$, and for any sufficiently large $\lambda$ depending on $\eps_0$ and $(\bar{v},\bar{p},\bar{R})$, we  have the following: There exists a smooth solution $(v,p,{R})$ of \eqref{e:ER}  satisfying the estimates
\begin{align*}
 \|v-\bar{v}\|_{H^{-1}} 
  &\leq C \lambda^{-1}\\
  \|v\|_{0}+\lambda^{-1}\|v\|_{1}
  &\leq C \\
  \norm{v\otimes v+ R -\bar{v}\otimes\bar{v}-\bar{R}}_{H^{-1}}
  &\leq C\lambda^{\gamma-1}\\
  \|\mathring R \|_{0}
  &\leq C \lambda^{\gamma-1} \\
   \norm{\tr R}_0  &\leq \eps \, ,
\end{align*}
where $C$ depends solely on $(\bar{v},\bar{p},\bar{R})$, and $\mathring{R}$ is the traceless part of $R$.  Moreover setting
\begin{align}
  e(t):=\int_{\T^3}|\bar{v}|^2+\tr\bar{R} \,dx
\label{e:h_e}
\end{align}
for any $t\in [0,T]$ we have
\[\frac{\eps}{2}\leq   e(t) - \int_{\T^3} |v|^2 dx \leq\eps\,.\]
\end{theorem} 

\begin{proof}[Proof of Theorem \ref{t:h-principle}]
Fix $k \geq 1$ and let $\eps_k<\eps_0$. We apply Theorem \ref{t:first_step} with $\gamma=\alpha$ and $\lambda=\lambda_0$, where here $(\alpha,\lambda_0)$ are given in the statement of Proposition \ref{p:main},  and where we take $a$ sufficiently large such that $\lambda_0$ is sufficiently large (in terms of $\eps_k$ and $(\bar v, \bar p, \bar R)$), so that the hypothesis of Theorem \ref{t:first_step} is satisfied. We obtain $( {v}, {p}, {R})$ satisfying
\begin{align}
 \|v-\overline v\|_{H^{-1}}\leq &C \lambda_0^{-1} \label{e:h_v_diff}\\
 \|v\|_{0}+\lambda_0^{-1}\|v\|_{1}\leq &C \label{e:h_v_D} \\
 \norm{v\otimes v+ R -\bar{v}\otimes\bar{v}-\bar{R}}_{H^{-1}}
  \leq& C\lambda_0^{\alpha-1} \label{e:h_vov_diff}\\
 \|\mathring R \|_{0}\leq &C \lambda_0^{\alpha-1} \label{e:h_R_circ} \\
 \norm{ \tr R}_0 \leq& \eps_k \label{e:h_R_tr} \, ,
 \end{align}
 and the function $  e(t)$ as defined by \eqref{e:h_e} obeys
 \begin{align}
\frac{\eps_k}{2} \leq     e(t)-\int_{\T^3}\abs{  v}^2\,dx  \leq  \eps_k\,.
\label{e:tilde_e_est}
\end{align}

Analogous to the proof of Theorem \ref{t:energy_profile}, we set
\[\Gamma=\frac{\delta_1^{\sfrac{1}{2}}}{\eps_k^{\sfrac 12}}\]
and rescale $( {v}, {p}, {R})$ to obtain
\begin{align*}
\widetilde v_0(x,t):=\Gamma v(x,\Gamma t ),\qquad \widetilde p_0(x,t):=\Gamma^2p  (x,\Gamma t )\quad\mbox{ and}\qquad
   {\widetilde R_0}(x,t):=\Gamma^2  R (x,\Gamma t )\,,
\end{align*}
so that $(v_0,p_0, {R}_0)$ also solves \eqref{e:ER}. Moreover, we have the estimates
\begin{align}
  \|\widetilde v_0\|_{0}+\lambda_0^{-1}\|\widetilde v_0\|_{1}\leq &\frac{C \delta_1^{\sfrac12}}{\eps_k^{\sfrac12}}  \label{e:h_tilde_v0} \\
  \|\mathring{\widetilde R_0} \|_{0}\leq &\frac{C \delta_1}{\eps_k \lambda_0^{1-\alpha}} \,. \notag
\end{align}
Choosing $\alpha$ sufficiently small and choosing $a$ sufficiently large depending on $\eps_k$, $C$, and $M$, we obtain
\[
\frac{C \delta_1^{\sfrac12}}{\eps_k^{\sfrac12}}\leq \min(M\delta_0^{\sfrac12},1-\delta_0) \quad\mbox{and}\quad \frac{C}{\eps_k \lambda_0^{1-\alpha}} \leq \lambda_0^{-3\alpha}\,.
\]
from which we obtain \eqref{e:R_q_inductive_est}, \eqref{e:v_q_inductive_est}, and \eqref{e:v_q_0}.

If in addition we set
\[
\tilde e(t)=\Gamma^2   e(\Gamma t)\,
\]
then from \eqref{e:tilde_e_est} we obtain
\[\frac{\delta_1}{2} \leq   \tilde e(t)-\int_{\T^3}\abs{\widetilde  v_0}^2\,dx\leq \delta_1\,,\]
and hence we obtain \eqref{e:energy_inductive_assumption} for $q=0$. Letting $a$ be sufficiently large, we also obtain \eqref{e:energy_time_D}.

Applying Proposition  \ref{p:main} and arguing as was done in the proof of Theorem \ref{t:energy_profile} we obtain a solution $(\widetilde v ,\widetilde p )$ to the Euler equations satisfying 
\begin{align}
\int_{\T^3}|\widetilde v |^2\,dx=\tilde e(t)\,.
\label{e:h_ess}
\end{align}
Moreover, by \eqref{e:v_diff_prop_est} we have the estimate
\begin{align}
\norm{\widetilde v - \widetilde v_0}_0\lesssim \delta_1^{\sfrac12}\,.
\label{e:h_v_diff2}
\end{align}

Lastly, we define $(v_k, p_k)$ by the rescaling
\[
v_k:=\Gamma^{-1}\widetilde v (x,\Gamma^{-1}t)\quad\mbox{and}\quad p_k:=\Gamma^{-2}\widetilde p (x,\Gamma^{-1}t)\,.
\]
Then $(v_k,p_k)$ is a solution to the Euler equations, satisfying \eqref{e:h-principle_energy} as a consequence of rescaling \eqref{e:h_ess}. The sequence $v_k$ is uniformly bounded in $C^0$ since
\begin{align*}
\norm{v_k}_0 \leq \Gamma^{-1} ( \norm{\widetilde v}_0 + \norm{\widetilde v - \widetilde v_0}_0) \lesssim \eps_k^{\sfrac 12} \delta_{1}^{-\sfrac 12} ( \delta_1^{\sfrac 12} + C \delta_1^{\sfrac 12} \eps_k^{-\sfrac 12}) \lesssim \eps_0^{\sfrac 12} + C.
\end{align*}
Thus $(v_k \otimes v_k)$ is also uniformly bounded in $C^0$. By Banach-Alaoglu $v_k$ and $v_k \otimes v_k$ have weak$-*$ convergent subsequences.

Moreover,  by rescaling \eqref{e:h_v_diff2} and using \eqref{e:h_v_diff} we have
\begin{align}
\|v_k-\overline v\|_{H^{-1}}\lesssim \norm{v_k - v }_{0} + \norm{v - \overline v}_{H^{-1}} \lesssim \Gamma^{-1} \delta_1^{\sfrac 12} + C \lambda_0^{-1} \lesssim  \eps_k^{\sfrac 12} + C \lambda_0^{-1} \lesssim   \eps_k^{\sfrac 12}
\label{e:h_v_H-1}
\end{align}
by choosing $a$ (and thus $\lambda_0$) sufficiently large in terms of $\eps_k$.
Moreover, from \eqref{e:h_vov_diff}--\eqref{e:h_R_tr}, \eqref{e:h_tilde_v0}, and \eqref{e:h_v_diff2} we obtain
\begin{align}
\norm{v_k \otimes v_k - v \otimes v -\bar R }_{H^{-1}} 
&\lesssim \norm{v_k \otimes v_k - v\otimes v}_{0} + \norm{R}_{0} + \norm{v\otimes v + R - \bar v \otimes \bar v - \bar R}_{H^{-1}} \notag\\
&\lesssim\Gamma^{-2} \norm{\widetilde v \otimes \widetilde v - \widetilde v_0 \otimes \widetilde v_0}_{0}  + \norm{\mathring R}_{0}  + \norm{\tr R}_0 + C \lambda_0^{\alpha-1} \notag\\
&\lesssim\eps_k \delta_1^{-1/2}  (C \delta_{1}^{\sfrac 12} \eps_k^{- \sfrac 12} + \delta_1^{\sfrac 12})+ \eps_k +  C \lambda_0^{\alpha-1}
\lesssim C \eps_k^{\sfrac 12}.
\label{e:h_vov_H-1}
\end{align}
Since the $H^{-1}$ topology uniquely captures the weak$-*$ limit, the theorem is completed upon passing $\eps_k \to 0$ in \eqref{e:h_v_H-1}--\eqref{e:h_vov_H-1}.
\end{proof}

\appendix

\section{H\"older spaces}\label{s:hoelder}

In the following $m=0,1,2,\dots$, $\alpha\in (0,1)$, and $\theta$ is a multi-index. We introduce the usual (spatial) 
H\"older norms as follows.
First of all, the supremum norm is denoted by $\|f\|_0:=\sup_{\T^3\times [0,1]}|f|$. We define the H\"older seminorms 
as
\begin{equation*}
\begin{split}
[f]_{m}&=\max_{|\theta|=m}\|D^{\theta}f\|_0\, ,\\
[f]_{m+\alpha} &= \max_{|\theta|=m}\sup_{x\neq y, t}\frac{|D^{\theta}f(x, t)-D^{\theta}f(y, t)|}{|x-y|^{\alpha}}\, ,
\end{split}
\end{equation*}
where $D^\theta$ are {\em space derivatives} only.
The H\"older norms are then given by
\begin{eqnarray*}
\|f\|_{m}&=&\sum_{j=0}^m[f]_j\\
\|f\|_{m+\alpha}&=&\|f\|_m+[f]_{m+\alpha}.
\end{eqnarray*}
Moreover, we will write $[f (t)]_\alpha$ and $\|f (t)\|_\alpha$ when the time $t$ is fixed and the
norms are computed for the restriction of $f$ to the $t$-time slice.

Recall the following elementary inequalities:
\begin{equation}\label{e:Holderinterpolation}
[f]_{s}\leq C\bigl(\varepsilon^{r-s}[f]_{r}+\varepsilon^{-s}\|f\|_0\bigr)
\end{equation}
for $r\geq s\geq 0$, $\eps>0$, and 
\begin{equation}\label{e:Holderproduct}
[fg]_{r}\leq C\bigl([f]_r\|g\|_0+\|f\|_0[g]_r\bigr)
\end{equation}
for $r\geq 0$. From \eqref{e:Holderinterpolation} with $\eps=\|f\|_0^{\sfrac{1}{r}}[f]_r^{-\sfrac{1}{r}}$ we obtain the 
standard interpolation inequalities
\begin{equation}\label{e:Holderinterpolation2}
[f]_{s}\leq C\|f\|_0^{1-\sfrac{s}{r}}[f]_{r}^{\sfrac{s}{r}}.
\end{equation}

Next we collect two classical estimates on the H\"older norms of compositions. These are also standard, for instance
in applications of the Nash-Moser iteration technique (for a detailed proof the reader might consult \cite[Proposition 4.1]{DlSzJEMS}).

\begin{proposition}\label{p:chain}
Let $\Psi: \Omega \to \mathbb R$ and $u: \R^n \to \Omega$ be two smooth functions, with $\Omega\subset \R^N$. 
Then, for every $m\in \mathbb N \setminus \{0\}$ there is a constant $C$ (depending only on $m$,
$N$ and $n$) such that
\begin{align}
\left[\Psi\circ u\right]_m &\leq C ([\Psi]_1 \|Du\|_{m-1}+\|D\Psi\|_{m-1} \|u\|_0^{m-1} \|u\|_m)\label{e:chain0}\\
\left[\Psi\circ u\right]_m &\leq C ([\Psi]_1 \|Du\|_{m-1}+\|D\Psi\|_{m-1} [u]_1^{m} )\, .
\label{e:chain1}
\end{align} 
\end{proposition}

We also recall the quadratic commutator estimate of~\cite{CoETi1994} 
(cf. also \cite[Lemma 1]{CoDLSz2012}):

\begin{proposition}\label{p:CET}
Let $f,g\in C^{\infty}(\T^3\times\T)$ and $\psi$ a standard radial smooth and compactly supported kernel. For any $r\geq 0$  we have the estimate
\[
\Bigl\|(f*\psi_\ell)( g*\psi_\ell)-(fg)*\psi_\ell\Bigr\|_r\leq C\ell^{2-r}  \|f\|_1\|g\|_1 \, ,
\]
where the constant $C$ depends only on $r$.
\end{proposition}

\section{Estimates for transport equations}\label{s:transport_equation}

In this section we recall some well known results regarding smooth solutions of
the \emph{transport equation}:
\begin{equation}\label{e:transport}
\left\{\begin{array}{l}
\partial_t f + v\cdot  \nabla f =g,\\ \\
f(\cdot,0)=f_0,
\end{array}\right.
\end{equation}
where $v=v(t,x)$ is a given smooth vector field. 
We will consider solutions
on the entire space $\R^3$ and treat solutions on the torus simply as periodic solution in $\R^3$. The following proposition contains standard estimates for such solutions (for a detailed proof, the reader might consult \cite[Appendix D]{BuDLeIsSz2015}).

\begin{proposition}\label{p:transport_derivatives}
Assume $\abs{t}\norm{v}_1\leq 1$. Then,  any solution $f$ of \eqref{e:transport} satisfies 
\begin{align}
\|f (t)\|_0 &\leq \|f_0\|_0 + \int_{t_0}^t \|g (\cdot, \tau)\|_0\, d\tau\,,\label{e:max_prin}\\
\norm{f(t)}_{\alpha} &\leq 2\left(\norm{f_0}_{\alpha} + \int_{t_0}^t  \norm{g (\cdot, \tau)}_{\alpha}\, d\tau\right)\,,\label{e:trans_est_alpha}
\end{align}
for all $0\leq \alpha\leq 1$, and, more generally, for any $N\geq 1$ and $0\leq \alpha< 1$
\begin{align}
[f (t)]_{N+\alpha} & \lesssim [ f_0]_{N+\alpha} + \abs{t}[v]_{N+\alpha}[f_0]_1  +\int_{0}^t \Bigl([g (\tau)]_{N+\alpha} + (t-\tau) [v ]_{N+\alpha} [g (\tau)]_{1}\Bigr)\,d\tau.
\label{e:trans_est_1}
\end{align}
Define $\Phi (t, \cdot)$ to be the inverse of the flux $X$ of $v$ starting at time $t_0$ as the identity
(i.e. $\sfrac{d}{dt} X = v (X,t)$ and $X (x, t_0 )=x$). Under the same assumptions as above we have:
\begin{align}
\norm{\nabla\Phi (t) -\Id}_0&\lesssim \abs{t}[v]_1\,,  \label{e:Dphi_near_id}\\
[\Phi (t)]_N &\lesssim \abs{t}[v]_N \qquad \forall N \geq 2\, .\label{e:Dphi_N}
\end{align}
\end{proposition}

\section{Potential theory estimates}
We recall the definition of the standard class of periodic Calder{\'o}n-Zygmund operators.
Let $K$ be an $\R^3$ kernel which obeys the properties
\begin{itemize}
\item $ K(z) = \Omega\left(\frac{z}{|z|}\right) |z|^{-3} $, for all $z\in\R^3 \setminus \{0\}$ 
\item $\Omega \in C^\infty({\mathbb S}^2)$
\item $\int_{|\hat z|=1} \Omega(\hat z) d\hat z = 0$.
\end{itemize}
From the $\R^3$ kernel $K$, use Poisson summation to define the periodic kernel 
\begin{align*}
K_{\T^3}(z) =  K(z) + \sum_{\ell \in {\mathbb Z}^3 \setminus \{0\}} \left( K(z+\ell) - K(\ell) \right).
\end{align*}
Then the operator
\begin{align*}
T_K f(x) = p.v. \int_{\T^3} K_{\T^3}(x-y) f(y) dy
\end{align*}
is a $\T^3$-periodic Calder{\'o}n-Zygmund operator, acting on $\T^3$-periodic functions $f$ with zero mean on $\T^3$.
The following proposition, proving the boundedness of periodic Calder{\'o}n-Zygmund operators on periodic H\"older spaces is classical (see e.g.~\cite{CaZy1954}):
\begin{proposition}
\label{p:CZO_C_alpha}
Fix $\alpha \in (0,1)$. Periodic Calder{\'o}n-Zygmund operators are bounded on the space of zero mean $\T^3$-periodic $C^\alpha$ functions. 
\end{proposition}
 
The following is a simple consequence of classical stationary phase techniques. For a detailed proof the reader might consult \cite[Lemma 2.2]{DaSz2016}.
\begin{proposition}
\label{p:R:oscillatory_phase}
Let  $\alpha \in(0,1)$ and $N \geq 1$. Let $a \in C^\infty(\T^3)$, $\Phi\in C^\infty(\T^3;\R^3)$ be smooth functions and assume that
\begin{equation*}
\hat{C}^{-1}\leq \abs{\nabla \Phi} \leq \hat{C}
\end{equation*}
holds on $\T^3$. Then 
\begin{equation}\label{e:phase_integration}
\abs{\int_{\T^3}a(x)e^{ik\cdot \Phi}\,dx}\lesssim  \frac{ \norm{a}_{N}+\norm{a}_{0}\norm{\Phi}_{N}}{|k|^{N}} \, ,
\end{equation}
and for the operator $\RR$ defined in \eqref{e:R:def}, we have
 \begin{align*}
\norm{ {\mathcal R} \left(a(x) e^{i k\cdot \Phi} \right)}_{\alpha} 
&\lesssim \frac{  \norm{a}_{0}}{|k|^{1-\alpha}} +   \frac{ \norm{a}_{N+\alpha}+\norm{a}_{0}\norm{\Phi}_{N+\alpha}}{|k|^{N-\alpha}} \, ,
\end{align*}
where the implicit constant depends on $\hat{C}$, $\alpha$ and $N$, but not on $k$.
\end{proposition}

\section{Commutators involving singular integrals}
The following lemma is a variant of Lemma 1 from \cite{Co2015}: 
\begin{proposition} 
\label{p:com:CZ:multiplication}
Let $\alpha \in(0,1)$ and $N\geq 0$. Let $T_K$ be a Calder{\'o}n-Zygmund operator with kernel $K$. Let $b \in C^{N+1,\alpha}(\T^3)$ a vectorfield. Then we have
\begin{align*}
\norm{ [T_K , b\cdot \nabla] f }_{N+\alpha} \lesssim \norm{b}_{1+\alpha} \norm{f}_{N+\alpha}+\norm{b}_{N+1+\alpha} \norm{f}_{\alpha}
\end{align*}
for any $f \in C^{N+\alpha}(\T^3)$, where the implicit constant depends on $\alpha, N$ and $K$.

\end{proposition}

\begin{proof}[Proof of Proposition~\ref{p:com:CZ:multiplication}]
The case $N=0$ is precisely Lemma 1 in \cite{Co2015}, except that in the former paper, the proof is given for Calder{\'o}n-Zygmund operators defined on $\R^3$, and for functions in $C^\alpha(\R^3) \cap L^p(\R^3)$. However, note that if $f$ is the $1$-periodic extension to all of $\R^3$ of the function $f$ on $\T^3$, and if $\chi(y)$ is a smooth cutoff function, which is identically $1$ on $[-1-1/20,1+1/20]^3$, and vanishes on the complement of  $[-1-1/10,1+1/10]^3$, we then have that 
\begin{align*}
T_K f(x) &= p.v.\int_{\R^3} K(x-y) \chi(y) f(y) dy + T_{\rm smooth} f(x)
\end{align*}
where 
\begin{align*}
T_{\rm smooth} \colon C^0(\T^3) \to C^{N}(\T^3)
\end{align*}
is a bounded operator, for any $N \geq 0$. Thus, modulo using the smoothing property of $T_{\rm smooth}$, we may apply directly the proof in \cite{Co2015} to the periodic case of this paper.

Let us now consider the case $N\geq 1$, and to this end let $\theta$ be a multi-index with $|\theta|=N$. Then, by the Leibniz rule
\begin{align*}
\partial^\theta[T_K,b\cdot\nabla]f&=T_K(\partial^\theta(b\cdot \nabla f))-\partial^\theta(b\cdot\nabla T_Kf)\\
&=\sum_{\theta'}\binom{\theta}{\theta'}\Bigl\{T_K(\partial^{\theta'} b\cdot\nabla \partial^{\theta-\theta'}f)-\partial^{\theta'} b\cdot\nabla \partial^{\theta-\theta'}T_Kf\Bigr\}\\
&=\sum_{\theta'}\binom{\theta}{\theta'}\Bigl\{[T_K,\partial^{\theta'} b\cdot\nabla]\partial^{\theta-\theta'}f\Bigr\}\,.	
\end{align*}
Therefore we obtain from the case $N=0$:
$$
\norm{\partial^\theta[T_K,b\cdot\nabla]f]f}_\alpha\lesssim \sum_{j=0}^N\|b\|_{j+1+\alpha}\|f\|_{N-j+\alpha}.
$$
Furthermore, by interpolation
\begin{align*}
\|b\|_{j+1+\alpha}\lesssim \|b\|_{1+\alpha}^{1-\sfrac{j}{N}}\|b\|_{N+1+\alpha}^{\sfrac{j}{N}}, \quad \mbox{and} \quad 
\|f\|_{N-j+\alpha}\lesssim \|f\|_{N+\alpha}^{1-\sfrac{j}{N}}\|f\|_{\alpha}^{\sfrac{j}{N}}\, ,
\end{align*}
so that, for any $j=0,\dots,N$
$$
\|b\|_{j+1+\alpha}\|f\|_{N-j+\alpha}\lesssim \norm{b}_{1+\alpha} \norm{f}_{N+\alpha}+\norm{b}_{N+1+\alpha} \norm{f}_{\alpha}\,.
$$
This concludes the proof.
\end{proof}



\end{document}